\documentclass{amsart}

\usepackage[bookmarks,colorlinks=true,citecolor=OliveGreen,linkcolor=RoyalBlue]{hyperref}

\usepackage{amssymb,amsthm}
\usepackage{mathtools}
\usepackage{mathabx}
\usepackage[usenames,dvipsnames]{xcolor}
\usepackage{thmtools}

\usepackage{color,soul}
\definecolor{ItalianApricot}{rgb}{1,0.7,0.5}
\sethlcolor{ItalianApricot}

\usepackage{tikz}
\usetikzlibrary{shapes}
\usetikzlibrary{shadows}
\usetikzlibrary{shapes.multipart}

\usepackage[capitalise]{cleveref}		

\makeatletter
\def\thmt@refnamewithcomma #1#2#3,#4,#5\@nil{%
  \@xa\def\csname\thmt@envname #1utorefname\endcsname{#3}%
  \ifcsname #2refname\endcsname
    \csname #2refname\expandafter\endcsname\expandafter{\thmt@envname}{#3}{#4}%
  \fi
}
\makeatother

\declaretheorem[numberwithin=section]{theorem}
\declaretheorem[sibling=theorem]{proposition}

\declaretheorem[sibling=theorem,style=definition]{definition}

\declaretheorem[sibling=theorem,style=remark]{lemma}
\declaretheorem[sibling=theorem,style=remark]{remark}
\declaretheorem[sibling=theorem,style=remark]{notation}
\declaretheorem[sibling=theorem,style=remark,refname={Porism,Porisms}]{porism}
\declaretheorem[numberwithin=theorem,style=remark,refname={Claim,Claims},Refname={Claim,Claims}]{claim}

\newcommand{\seq}[1]{{\left\langle{#1}\right\rangle}}
 
\newcommand{\rest}[1]{\! \upharpoonright_{#1}} 
\newcommand{\tth}{{}^{\textup{th}}}
\newcommand{\conc}{\hat{\,\,}}

\newcommand{\Then}{\,\Longrightarrow\,}

\newcommand{\then}{\,\,\rightarrow\,\,}

\renewcommand{\setminus}{-}

\newcommand{\+}[1]{{\boldsymbol{#1}}}

\newcommand{\w}{\omega}
\newcommand{\s}{\sigma}

\renewcommand{\epsilon}{\varepsilon}

\renewcommand{\le}{\leqslant}
\renewcommand{\ge}{\geqslant}
\renewcommand{\leq}{\leqslant}
\renewcommand{\geq}{\geqslant}
\renewcommand{\preceq}{\preccurlyeq}
\renewcommand{\succeq}{\succcurlyeq}
\newcommand{\nle}{\nleqslant}

\newcommand{\boldface}[1]{\boldsymbol{#1}}

\newcommand{\Tur}{\textup{\scriptsize T}}
\newcommand{\ML}{\textup{ML}}
\newcommand{\MLR}{\textup{\texttt{MLR}}}
\newcommand{\leb}{\lambda}

\newcommand{\degd}{\mathbf{d}}

\newcommand{\wock}{\w_1^{\textup{ck}}}

\newcommand{\hT}{\wock{}\Tur}

\newcommand{\UU}{\mathbb U}

\newcommand{\U}{\mathcal{U}}

\newcommand{\QQ}{\mathbb{Q}}

\newcommand{\symdiff}{\!\!\vartriangle\!\!}

\newcommand{\Rat}{{\mathbb Q}}

\newcommand{\given}{\,\,|\,\,}

\newcommand{\EE}{\mathcal E}
\newcommand{\cost}{\+{c}}

\title{Continuous higher randomness}

\author{Laurent Bienvenu}
\address{LIAFA, Université Paris Diderot, Paris, France}
\email{laurent.bienvenu@computability.fr}
\urladdr{\url{http://www.liafa.jussieu.fr/~lbienven/}}

\author{Noam Greenberg} 
\address{School of Mathematics, Statistics and Operations Research \\ Victoria University of Wellington \\ Wellington, New Zealand}
\email{greenberg@msor.vuw.ac.nz}
\urladdr{\url{http://homepages.mcs.vuw.ac.nz/~greenberg/}}

\author{Benoit Monin}
\address{School of Mathematics, Statistics and Operations Research \\ Victoria University of Wellington \\ Wellington, New Zealand}
\email{benoit.monin@computability.fr}
\urladdr{\url{http://www.liafa.univ-paris-diderot.fr/~benoitm/}}

\thanks{Greenberg was supported by the Marsden Fund, by a Rutherford Discovery Fellowship from the Royal Society of New Zealand, and by a Turing research fellowship from the Templeton Foundation. Monin was partially supported by the Marsden Fund.}

\begin{document}

\begin{abstract}
We investigate the role of continuous reductions and continuous relativisation in the context of higher randomness. We define a higher analogue of Turing reducibility and show that it interacts well with higher randomness, for example with respect to van-Lambalgen's theorem and the Miller-Yu / Levin theorem. We study lowness for continuous relativization of randomness, and show the equivalence of the higher analogues of the different characterisations of lowness for Martin-L\"of randomness. We also characterise computing higher $K$-trivial sets by higher random sequences. We give a separation between higher notions of randomness, in particular between higher weak-2-randomness and $\Pi^1_1$-randomness. To do so we investigate classes of functions computable from Kleene's~$O$ based on strong forms of the higher limit lemma. 
\end{abstract}

\maketitle

\section{Introduction}

Algorithmic randomness uses the tools of computability theory to give a formal definition of the notion of a random infinite binary sequence, a sequence we would expect be the result of independent coin tosses. Many theorems of probability theory and analysis detail properties of real numbers which are shared by all elements of a set of measure~$1$. In other words a ``typical'' -- or ``random" real satisfies the property. For example, a monotone function is differentiable at almost every real. This fact though does not tell us what ``typical reals'' are; for every real~$x$ there is some monotone function which is not differentiable at~$x$. Restricting ourselves to a computable viewpoint allows us to consider only countably many properties of measure~$1$. For example we can characterise the collection of reals~$x$ at which every computable monotone function is differentiable \cite{BrattkaMillerNies2011}. 

Varying the computational strength of the tools involved we obtain in fact a hierarchy of randomness notions. Roughly, the stronger the tools we have the easier it is to detect irregular behaviour and so the harder it is to be considered random. Many of the resulting notions of randomness are robust. The best known notion, due to Martin-L\"of \cite{Martin-Lof1966}, can be defined by using computably enumerable betting strategies, by the incompressibility of initial segments, and by specifying a natural class of effectively presented, effectively null~$G_\delta$ sets. The resulting field studies these notions of randomness, investigates questions such as ``what does it mean for one sequence to be more random than another?'', measures the computational strength of random oracles, looks at connections to effective analysis, and much more (see \cite{Nies2009, DowneyH2010}). A particularly deep area of investigation concerns notions opposite to randomness, such as $K$-triviality, and relates them to computational weakness. 

While they in some sense formalise the intuitive notion of effective computation (albeit disregarding questions of time and space resources), computability-related notions do not satisfy natural closure properties. For example, the variation function of a computable function of bounded variation need not be computable. As result, even though every function of bounded variation is the difference of two monotone functions, a real number~$x$ can be random in the sense that every computable monotone function is differentiable at~$x$, but not in the sense that every computable function of bounded variation is differentiable at~$x$. This is related to the fact that the halting problem is not computable. To overcome similar problems, Martin-L\"of himself suggested that the ``pattern detection tools'' for defining randomness should be taken from a much larger collection. Such collections are given by the closely-related fields of effective descriptive set theory and so-called ``higher computability'' (see \cite{Sacks1990}). The collection of~$\Delta^1_1$ (or \emph{hyperarithmetic}) sets is the smallest one closed under taking the relativised halting problem and closing downward under Turing reducibility; alternatively, under taking infinite computable Boolean operations. Martin-L\"of defined a real to be \emph{$\Delta^1_1$-random} if it is an element of every~$\Delta^1_1$ set of measure~$1$. The closure properties of the hyperarithmetic sets result, for example, in the fact that a real~$x$ is~$\Delta^1_1$ random if and only if every~$\Delta^1_1$ monotone function is differentiable at~$x$ if and only if every~$\Delta^1_1$ function of bounded variation is differentiable at~$x$.

Beyond the desirable closure properties, working with~$\Delta^1_1$ and~$\Pi^1_1$ sets is particularly natural and appealing to computability theorists. This is because one can view these notions as analogues of the fundamental and familiar notions of ``computable'' and ``computably enumerable'', interpreted over an enlarged domain of computation. The theory of admissible computability generalises computability to \emph{admissible ordinals}. The smallest admissible ordinal is~$\wock$, the least ordinal which is not the order-type of a computable well-ordering of the natural numbers. The corresponding domain of computation is~$L_{\wock}$, the smallest admissible set, which is the initial segment of the constructible universe of height~$\wock$. A real is~$\Delta^1_1$ if and only if it is an element of~$L_{\wock}$. The Spector-Gandy theorem says that the~$\Pi^1_1$ sets are those which are defined by an existential quantifier ranging over the collection of hyperarithmetic sets. 
Via coding of structures by reals this shows that the~$\Pi^1_1$ sets are precisely those which are computably enumerable over the structure~$L_{\wock}$. Informally, these are the sets that can be enumerated effectively if the enumeration procedure takes~$\wock$ many steps. With this viewpoint in mind, many intuitive ideas from traditional ``countable'' computability (computability over~$\w$), for example reduction and separation theorems (or the fixed-point theorem) extend to the higher setting with precisely the same proofs.

An important advance in the theory of ``higher randomness'' was made by Hjorth and Nies in \cite{HjorthNies2007}. They examined the higher analogue of Martin-L\"of randomness and also isolated the new, stronger notion of~$\Pi^1_1$-randomness. They also looked at the higher analogues of the~$K$-trivial sets. The theory was then further developed by Chong, Nies and Yu \cite{ChongNiesYu2008} and by Chong and Yu \cite{ChongYu}. One of the projects they are concerned with is the separation of higher notions of randomness. One of the results in this paper is the separation between~$\Pi^1_1$-randomness and the higher analogue of weak 2-randomness. We also consider higher $K$-triviality.

\subsection{Randomness and continuity}

A main theme of this paper is the centrality of continuous reductions to the theory of randomness. The insight that randomness and traditional relative hyperarithmetic reducibility do not interact well goes back to Hjorth and Nies \cite{HjorthNies2007}.

As a first motivating example we consider the fact that strong randomness notions are downward closed in the Turing degrees of \ML-random sets. For example, Miller and Yu \cite{MillerY2008} showed that if~$X$ and~$Y$ are \ML-random,~$Y$ computes~$X$ and~$Y$ is in addition $\degd$-random (for some Turing degree~$\degd$) then~$X$ too is~$\degd$-random. Similarly, an \ML-random set~$X$ is weakly 2-random if and only if it forms a minimal pair with~$\emptyset'$ \cite[Theorem 5.3.15]{Nies2009}, a property clearly downward closed in the Turing degrees. Another example is difference randomness, which is equivalent to being \ML-random and not computing~$\emptyset'$ \cite{Franklin2011}. The argument of Miller and Yu's works for almost every randomness notion stronger than Martin-L\"of's: suppose that~$Y$ computes~$X$ and that~$X$ is random; say $\Phi(Y)=X$ where~$\Phi$ is some Turing functional. For a finite binary string~$\s$ let $\Phi^{-1}[\s]$ be the collection of oracles~$Z$ such that $\Phi(Z)\succeq \s$; we include oracles for which $\Phi(Z)$ is not total. Then $\s\mapsto \leb(\Phi^{-1}[\s])$ (here~$\leb$ denotes Lebesgue measure on Cantor space~$2^\w$) is a continuous c.e.\ semimeasure (multiplied by~$2^{|\s|}$ it is a c.e.\ supermartingale). Since~$X$ is \ML-random, $\leb (\Phi^{-1}[X\rest n])\le^\times 2^{-n}$. By withholding computations, we can massage the functional~$\Phi$ so that $\leb(\Phi^{-1}[\s])\le^\times 2^{-|\s|}$ for all~$\s$ (but still~$\Phi(Y)=X$). Using the massaged functional we can pull back any strong test~$\seq{U_n}$ which captures~$X$ (a difference test, a weak 2-test, a Demuth test, a $\degd$-\ML-test,...) and obtain a similar test which captures~$Y$. 

The key to this argument is the continuity of the map~$\Phi$ on~$2^\w$. The reducibility~$\le_h$ (relatively hyperarithmetic) is not given by partial continuous functions. And indeed some of the examples above fail in the higher setting. Hjorth and Nies \cite{HjorthNies2007} introduced the notion of $\Pi^1_1$-\ML-randomness, the higher analogue of \ML-randomness; Nies \cite[9.2.17]{Nies2009} introduced the notion of strong $\Pi^1_1$-\ML-randomness, the higher analogue of weak 2-randomness, studied later by Chong and Yu \cite{ChongYu}. There are however reals~$X$ and~$Y$ such that $X\le_h Y$, $Y$ is strongly $\Pi^1_1$-\ML-random, and~$X$ is $\Pi^1_1$-\ML-random but not strongly so.

Rather than use~$\le_h$, we need a continuous higher analogue of Turing reducibility. For preciseness, recall that a functional is simply a set of pairs~$(\tau,\s)$ of finite binary strings. Looking forward, note that we do not require that the functional be consistent; we discuss this shortly. If~$\Phi$ is a functional then for any $X\in 2^{\le \w}$ (finite or infinite) we let 
\[ \Phi(X) = \bigcup \left\{ \s  \,:\, (\tau,\s)\in \Phi\text{ for some }\tau\preceq X \right\} .\]
For $X,Y\in 2^\w$, $X\le_\Tur Y$ if and only if $\Phi(Y)=X$ for some c.e.\ functional~$\Phi$. This motivates the following definition:

\begin{definition} \label{def:higher_Turing}
	Let~$X,Y\in 2^\w$. $X$ is \emph{higher Turing reducible} to~$Y$ if $\Phi(Y)=X$ for some~$\Pi^1_1$ functional~$\Phi$. We write $X\le_{\hT} Y$. 
\end{definition}

With this notion some of the familiar theorems mentioned above generalise to the higher setting. For example, we will show:

\begin{theorem} \label{thm:downward_closure_Strong-ML-randomness}
	Let~$X,Y$ be $\Pi^1_1$-\ML-random. Suppose that $X\le_{\hT}Y$ and that~$Y$ is in fact strongly $\Pi^1_1$-\ML-random. Then~$X$ too is strongly $\Pi^1_1$-\ML-random. 
\end{theorem}

We will also see, for example, that a $\Pi^1_1$-\ML-random set~$X$ is higher difference random if and only if $O\nle_{\hT} X$, where Kleene's~$O$ is the complete~$\Pi^1_1$ set of numbers. On the other hand, we will see that some results only partially generalise, or completely fail in the higher setting. For example, a Martin-L\"of random real is weak-2-random if and only it forms a minimal pair with $\emptyset'$, but we show in \cite{Pi11RandomnessPaper} that it is not the case that a $\Pi^1_1$-ML-random set is strongly $\Pi^1_1$-ML-random if and only if it forms a minimal pair with Kleene's~$O$ in the $\le_{\hT}$-degrees. 

\medskip

Continuity also matters when it comes to relativizing randomness notions. The two-step product theorem (see for example \cite{Jech2008} or \cite{Kunen2011}) says that if~$\mathbb P$ and~$\mathbb Q$ are notions of forcing then a filter $G\times H \subset \mathbb P\times \mathbb Q$ is $V$-generic if and only if the filter $G\subset \mathbb P$ is $V$-generic and $H\subset \mathbb Q$ is $V[G]$-generic. The theorem has effective analogues. For example, a join $G\oplus H$ is (Cohen) 1-generic if and only if~$G$ is 1-generic and~$H$ is 1-generic relative to~$G$ (see for example \cite{Yu2006}). van Lambalgen \cite{vanLambalgen1987} gave an analogous effectivisation for \ML-randomness. It fails in the higher setting: there are reals~$X$ and~$Y$ such that~$X\oplus Y$ is $\Pi^1_1$-\ML-random, but~$Y$ is not $\Pi^1_1(X)$-\ML-random. The reason for this failure is that the relativisation is not continuous: enumerating clopen subsets of a component of a $\Pi^1_1(X)$-\ML-test is not determined by only finitely many bits of~$X$. Similarly to Turing reducibility, we need to define a continuous higher analogue of being computably enumerable relative to an oracle. The treatment is similar. An \emph{enumeration functional} is a set of pairs $(\tau,m)$ consisting of a finite binary string and a natural number. If~$\Psi$ is an enumeration functional and $X\in 2^{\le \w}$ then we let
\[ \Psi^X =  \left\{ m  \,:\, (\tau ,m)\in \Psi \text{ for some }\tau\preceq X \right\}.\]
A set~$B$ is c.e.\ in~$X$ if and only if $B = \Psi^X$ for some c.e.\ enumeration functional~$\Psi$. 

\begin{definition} \label{def:higher_ce}
  Let $X\in 2^\w$. A set $B\subseteq \w$ is \emph{higher $X$-c.e.} if $B = \Psi^X$ for some~$\Pi^1_1$ enumeration functional~$\Psi$.\footnote{We remark that we can think of an enumeration functional as an open subset of~$2^\w\times \w$. If~$\Psi$ is such a set then $\Psi^X$ is the $X$-section of~$\Psi$.}
\end{definition}

Armed with this definition we can consider higher $X$-c.e.\ open sets (sets of the form $\bigcup_{\s\in B}[\s]$ where $B$ is a higher $X$-c.e.\ set of strings), and so higher $X$-\ML-tests and higher $X$-\ML-randomness. Thus $\Pi^1_1$-\ML-randomness is simply higher $\emptyset$-\ML-randomness, and so we call it ``higher \ML-randomness''. We will show that this continuous relativisation satisfies van Lambalgen's theorem.

\begin{theorem} \label{thm:van_Lambalgen}
	Let~$X,Y\in 2^\w$. Then $X\oplus Y$ is higher \ML-random if and only if~$X$ is higher \ML-random and~$Y$ is higher $X$-\ML-random.
\end{theorem}

The issue of continuous relativisation is directly related to the study of anti-randomness and lowness for randomness. A celebrated result of Nies's (together with work by Hirschfeldt, Nies and Stephan \cite{Nies2005,HirschfeldtNiesStephan2007}) is the coincidence of a number of classes, each formalising a notion of distance from randomness or weakness as an oracle in detecting randomness: the $K$-trivial sets; the sets which are low for \ML-randomness; the sets which are low for~$K$; and the sets which are a base for \ML-randomness. Hjorth and Nies \cite{HjorthNies2007} showed that this result fails in the higher setting: while there are sets which are higher $K$-trivial but not hyperarithmetic, every set which is low for $\Pi^1_1$-\ML-randomness is hyperarithmetic. (Higher $K$-triviality is defined using a $\Pi^1_1$ analogue of prefix-free Kolmogorov complexity.) Again this uses the fact that the relativisation of $\Pi^1_1$-\ML-randomness used in the definition of lowness for this notion is not continuous. We will show that using continuous relativisation the coincidence does hold:

\begin{theorem} \label{thm:higher_K-trivial_coincidence}
	The following are equivalent for $A\in 2^\w$:
	\begin{enumerate}
		\item $A$ is higher $K$-trivial.
		\item Every higher \ML-random set is also higher $A$-\ML-random.
		\item There is some higher $A$-\ML-random set~$X$ such that $A\le_{\hT}X$. 
	\end{enumerate}
\end{theorem}

We will also discuss lowness for~$K$.

\subsection{A general method for defining higher analogues}

The two examples we gave of higher analogues of basic concepts of computability (Turing reducibility and relative computable enumerability) follow a common method which is already implicit in the Chong-Yu work and which we will employ everywhere. We realise that the most fundamental concept of computability theory is computable enumerability. From it, all other notions can be derived: a partial computable function is one with c.e.\ graph, Turing reducibility is defined using c.e.\ functionals, etc. Recall again that a set of numbers is $\Pi^1_1$ if and only if it is~$\Sigma_1$-definable over~$L_{\wock}$; in the terminology of higher computability, it is \emph{$\wock$-c.e.} The method of obtaining higher analogues is to replace every instance of ``c.e.'' by ``$\wock$-c.e.''. As we observed, this means that ``higher \ML-randomness'' is the notion of $\Pi^1_1$-\ML-randomness defined by Hjorth and Nies; and ``higher weak 2-randomness'' is the notion of strong $\Pi^1_1$-\ML-randomness defined by Chong, Nies and Yu. It is only the basic notion of computable enumerability which is being modified; all other quantifiers range over the natural numbers (rather than~$\wock$), and unlike metarecursion theory, the objects studied are subsets of~$\w$ rather than subsets of~$\wock$. For example, a higher \ML-test is an $\w$-sequence of (uniformly) higher c.e.\ open sets, rather than a sequence of length~$\wock$. The fact though that the basic existential quantifier (the computable unbounded search) ranges over~$\wock$ motivates some of our notation (such as $\le_{\hT}$). 

\subsection{Continuity and its discontent}

Beyond the inherent interest in higher notions, the study of generalisations of computability sheds light on the familiar notions by separating concepts which ``accidentally'' coincide in usual computability. An example of such a phenomenon is directly related to the examples of the use of continuity that we discussed above. 

Consider the definition of higher Turing reducibility. The definition of Turing reducibility in terms of functionals usually imposes extra requirements of \emph{consistency} on the functional. Namely that if $(\tau,\s)$ and $(\tau',\s')$ are two ``axioms'' in the functional~$\Phi$ and $\tau$ and~$\tau'$ are compatible, then~$\s$ and~$\s'$ are compatible. Indeed, in \cite{HjorthNies2007} Hjorth and Nies introduce a continuous reducibility (which they denote by~$\le_{fin-h}$). Their definition is similar to \cref{def:higher_Turing} except that they require that the functional~$\Phi$ be the graph of an order-preserving function from strings to strings and moreover that its domain is closed under taking initial segments. In ``traditional'' (or ``countable'') computability this extra requirement creates no difficulty. Namely $X\le_{\Tur} Y$ if and only if $X=\Phi(Y)$ for some c.e.\  functional~$\Phi$ \footnote{If $\Phi$ is an inconsistent Turing functional and two inconsistent axioms in~$\Phi$ apply to an oracle~$Y$ then $\Phi(Y)\notin 2^\w$ and so~$Y$ does not compute anything with the functional~$\Phi$.} if and only if $X= \Phi(Y)$ for some consistent c.e.\ functional if and only if $X=\Phi(Y)$ for some c.e.\ functional satisfying the definition of Hjorth and Nies. We will show in \cite{BadOracles} that the higher analogues of the two first notions are distinct, while the second one coincide with the third one, but not uniformly. In this paper we will argue that among these two distinct reducibilities, the one given by \cref{def:higher_Turing} is the one which fits best with the general theory of higher randomness.

It may be instructive to see why the argument that traditionally these reducibilities are the same fails in the higher setting. To turn an arbitrary functional into a consistent one (without losing total computations), when an axiom $(\tau,\s)$ enters the functional at some stage~$s$, we consider all extensions of~$\tau$ of length~$s$, and map those among them to~$\s$ for which this does not introduce an inconsistency. This argument uses what we call a ``time trick'': the fact that the number of stages is the same as the length of the oracle, namely~$\w$. This equality fails in the higher setting, in which we still use oracles of length~$\w$ but effective constructions have~$\wock$ many stages. Thus any argument that relies on a time trick cannot be simply copied in the higher setting. In some cases the argument can be rectified (an example is the proof of the higher Kraft-Chaitin theorem by Hjorth and Nies). In other cases, such as the equivalence of the three definitions of Turing reducibility, in the higher setting the theorem fails.

To give evidence that \cref{def:higher_Turing} is more useful than other possible generalisations of Turing reducibility to the higher setting, consider for example one of the most basic properties of relative computability. The following is easily verified using arguments of general computability:

\begin{proposition} \label{prop:relative_c.e._and_relative_Turing}
	The following are equivalent for $X,Y\in 2^\w$:
	\begin{enumerate}
		\item $X\le_{\hT} Y$. 
		\item Both~$X$ and its complement are higher~$Y$-c.e.
	\end{enumerate}
\end{proposition}

The proposition fails if we replace $\le_{\hT}$ by its stricter variant. The difficulty is in the direction (2)$\Then$(1): suppose that $\Psi_1^Y = X$ and $\Psi_0^Y = \w\setminus X$. We build a functional $\Phi$ with the aim that $\Phi(Y)=X$. When we see strings~$\tau$ and~$\s$ such that $\Psi_0^\tau \supseteq \{n\,:\, \s(n)=0\}$ and $\Psi_1^\tau \supseteq \{ n\,:\, \s(n)=1\}$ we enumerate the axiom $(\tau,\s)$ into~$\Phi$. It is possible that for other oracles~$Z$, $\Psi_1^Z$ and~$\Psi_0^Z$ do not enumerate a set and its complement. But before we see this fact, at earlier stages, computations corresponding to such oracles~$Z$ appear to give a set and its complement --- inconsistent with~$\s$ --- and enumerate into~$\Phi$ axioms (with use extending~$\tau$ but incomparable with~$Y$) which are inconsistent with $(\tau,\s)$. The current stage may be infinite (a stage $s\in [\w,\wock)$), and so such an event could have happened arbitrarily close to~$Y$ (i.e.\ extending longer and longer initial segments~$\tau$ of~$Y$). Thus even if we take~$\tau$ to be an arbitrarily long initial segment of~$Y$, enumerating $(\tau,\s)$ into~$\Phi$ makes~$\Phi$ inconsistent; of course $\Phi(Y)$ does not contain inconsistencies. This shows how the time trick can fail bitterly. In~\cite{BadOracles} we show how to turn this situation around to prove, for example, that the two generalisations of Turing reducibility are distinct: there exists some $X,Y$ such that $X \le_{\hT} Y$ but not via a functional which is consistent everywhere. 

\medskip

The utility of \cref{def:higher_Turing} with respect to randomness is witnessed in \cref{thm:higher_K-trivial_coincidence} (in the notion of a higher base for randomness) and also in the example of difference randomness. The following theorem is the correct generalisation of a theorem of Franklin and Ng's. We delay the proper definition of ``higher $\w$-computably approximable'' to the next subsection.

\begin{theorem} \label{thm:higher_difference_randomness}
	The following are equivalent for a higher ML-random set~$X$:
	\begin{enumerate}
		\item $O\nle_{\hT} X$ (where~$O$ is Kleene's complete $\Pi^1_1$ set).
		\item $X$ avoids all nested tests of the form~$\seq{U_n\cap P}$ where $\seq{U_n}$ are uniformly higher effectively open, $P$ is higher effectively closed (a closed~$\Sigma^1_1$ set of reals), and $\leb(U_n\cap P)\le 2^{-n}$. 
		\item $X$ avoids all nested tests of the form~$\seq{W_{f(n)}}$ where~$W_e$ (for $e<\w$) is the~$e\tth$ higher effectively open set; $\leb(W_{f(n)})\le 2^{-n}$; $f$ is higher~$\w$-computably approximable, witnessed by~$\seq{f_s}_{s<\wock}$; and if $f_{s}(n)\ne f_{t}(n)$ then $W_{f_{s}(n)}$ and $W_{f_{t}(n)}$ are disjoint. 
	\end{enumerate}
\end{theorem}

The proof is the same as in \cite{Franklin2011}. We note where we use the fact that inconsistent functionals are allowed. In proving (1)$\Then$(3) the functional~$\Gamma$ which we build determines that $\Gamma(\tau)(n)= O_s(n)$ where~$s$ is the least such that $[\tau]\subseteq W_{f(n)}[s]$. On the elements of the Solovay test $\{ W_{f(n)}[s]\,:\, n\in O_{s+1}\setminus O_{s}\}$ this functional may be inconsistent. In fact, in~\cite{BadOracles} we show that there is a higher ML-random sequence which is higher Turing above~$O$, but is not fin-h above~$O$.

\medskip

Similarly, our definition of the relativisation of higher \ML-randomness runs into consistency problems when we try to construct a uniform universal test. Classically there is a sequence~$\seq{U_n}$ of enumeration operators such that for all~$Z$, $\seq{U_n^Z}$ is a universal $Z$-\ML-test. This fails in the higher setting. The point is that we cannot take a higher c.e.\ operator (a $\Pi^1_1$ enumeration functional)~$U$ and produce another such functional~$V$ such that $\leb(V^Z)\le \epsilon$ for all~$Z$ (for some fixed~$\epsilon$), and such that $U^Z = V^Z$ if $\leb(U^Z)\le \epsilon$. Again a time trick fails. In a sense it is a topological problem. In standard computability, at a finite stage~$s$ the collection of reals for which an axiom of~$U_s$ applies is clopen. When the axiom $(\tau,\s)$ enters~$U$ (indicating that $[\s]\subseteq U^Z$ for all $Z\in [\tau]$) we let $C = \{ Z\in 2^\w\,:\, \leb([\s]\cup V_s^Z)> \epsilon\}$; this set is clopen and so we can let~$V$ enumerate~$[\s]$ with oracles in the clopen set $[\tau]\setminus C$. In the higher setting, $C$ is open but may fail to be clopen, as~$s$ may be infinite. Indeed~$C$ could be dense. There may be no way to add~$[\s]$ to~$V^Z$ for reals~$Z$ outside~$C$ without making $\leb(V^Z)>\epsilon$ for some reals $Z\in C$. 

This is an issue we will need to monitor; in some cases we can find work-arounds to get analogues of ``lower'' results. In other cases this is impossible. In \cite{BadOracles} we not only show that there is no uniform universal oracle higher \ML-test; indeed we construct an oracle~$A$ for which there is no universal higher $A$-\ML-test. 

Similarly we can construct an oracle relative to which there is no optimal higher discrete c.e.\ semimeasure, and so an oracle relative to which higher prefix-free Kolmogorov complexity $K^A$ is not defined. Thus we need to modify the definition of ``higher low for~$K$'', to say that every higher discrete $A$-c.e.\ semimeasure is dominated by the optimal higher c.e.\ one. Of course if~$A$ is low for higher~$K$ then higher~$K^A$ exists, namely it is higher~$K$. We will show that this notion coincides with higher~$K$-triviality as well.

\subsection{The higher limit lemma} \label{section_higher_limit_lemma}

The analysis of functions approximable by hyperarithmetic functions corresponds to that given to~$\Delta^0_2$ sets and functions by Shoenfield's limit lemma. Here Kleene's~$O$ plays the role of the halting problem~$\emptyset'$. This analysis will help us separate notions of higher randomness.

Recall that a sequence $\seq{f_s}_{s<\wock}$ is $\wock$-computable if it is $\Sigma_1$-definable over $L_{\wock}$. Such a sequence is a \emph{$\wock$-approximation} of a function~$f\in \w^\w$ if for all~$n$ there is some $s<\wock$ such that $f_t(n) = f(n)$ for all $t\in [s,\wock)$). In \cref{subsec:higher_limit_lemma} we shall prove:

\begin{proposition} \label{prop:higher_limit_lemma}
	The following are equivalent for $f\in \w^\w$:
	\begin{enumerate}
		\item $f\le_{\hT} O$;
		\item $f\le_{\Tur} O$;
		\item $f$ has a $\wock$-computable approximation. 
	\end{enumerate}
\end{proposition}

Since a subset of $\w$ is c.e.\ in $O$ if and only if it is $\Sigma_2$ definable over $L_{\wock}$, the functions computable from Kleene's $O$ are the functions which are $\Delta_2$-definable over $L_{\wock}$. Thus we call such functions ``higher $\Delta^0_2$''. We will investigate subclasses of the collection of all higher $\Delta^0_2$ functions (such as the higher $\w$-c.a.\ functions which we define below). These classes are related to notions of randomness in two ways:

\smallskip

\noindent\textbf{A.} In the style of Demuth, we can use higher $\Delta^0_2$ functions to give indices for higher effectively open components of tests: tests of the form $\seq{W_{f(n)}}$ where~$W_e$ is the $e\tth$ \emph{higher} c.e.\ open set. The strongest such notion is higher $\MLR[O]$, for which we use all functions $f\le_{\Tur} O$. Unlike lower computability, this is strictly stronger than the higher version of weak 2-randomness; indeed strictly stronger than $\Pi^1_1$-randomness.

\smallskip

\noindent\textbf{B.} We can study $\Delta^0_2$ properties according to their approximability. In the lower setting this involves classes determined by bounding the number of mind changes. For example in \cite{FigueiraAll2012} Figueira et al.\ show that while there is a ML-random sequence with an approximation whose first $n$ bits change at most $2^n$ many times, no such random can be superlow. The general theme is that among random sequences, approximations with few changes correspond to computational strength. 

\smallskip

In the higher setting we identify a number of classes of functions lying between the higher $\w$-c.a.\ functions and all higher $\Delta^0_2$ functions. In some sense they too are described by conditions about how often the approximation changes. However these conditions are qualitative rather than quantitative. Thus these classes have no lower analogues. 

\begin{definition} \label{def:collapsing_approximation}
	Let $\seq{f_s}$ be a $\wock$-computable approximation of a higher $\Delta^0_2$ function $f$. For $n<\w$ let $s(n)$ be the least stage $s<\wock$ such that $f_s\rest n  = f\rest n$. The approximation $\seq{f_s}$ is \emph{collapsing} if $\sup_n s(n) = \wock$. Equivalently, $\seq{f_s}$ is collapsing if for all $s<\wock$, $f$ does not belong to the closure of the set $\{f_t\,:\, t<s\}$.
\end{definition}

Gandy's basis theorem implies that there is an $O$-computable $\Pi^1_1$-random sequence, and so a $\Pi^1_1$-random sequence with some $\wock$-computable approximation. However no such random sequence can have a collapsing approximation, since the sequence $\seq{s(n)}$ is $\Sigma_1$-definable over $(L_{\wock},f)$, and so if $f$ has a collapsing approximation then $\w_1^f > \wock$ ($f$ collapses $\wock$). Roughly, the intuition here is that an approximation of a $\Pi^1_1$-random sequence $X$ must change so much so that all initial segments of~$X$ appear long before the end of the approximation. We note though that there are sets $X\le_{\Tur} O$ which collapse $\wock$ but do not have a collapsing approximation. 

\medskip

Some of the classes we consider are defined by topological conditions. For example:

\begin{definition} \label{def:compact_approximation}
	A $\wock$-computable approximation $\seq{f_s}$ of a function~$f$ is \emph{compact} if the set $\{f_s\,:\, s<\wock\}\cup \{f\}$ is a compact subset of Baire space $\w^\w$.
\end{definition}

Of course if we approximate an element of Cantor space we may assume that all elements of the approximation are also elements of Cantor space. In that case an approximation is compact if and only if it is closed (for the usual topology). 

\begin{lemma} \label{lem:compact_implies_collapsing}
	Suppose that $\seq{f_s}$ is a compact approximation of a function $f\notin \Delta^1_1$. Then $\seq{f_s}$ is a collapsing approximation. 
\end{lemma}

\begin{proof}
	Let $s(n)$ be defined as above. Suppose that $s(\w) =\sup_n s(n)$ is a computable ordinal. Consider the closure $A$ of the set $\{f_t\,:\, t<s(\w)\}$. The function $f$ is an element of $A$. However $A$ is countable, as it is contained in the compact set $\{f_t\,:\, t<\wock\}\cup \{f\}$. Further, $A$ is the set of paths of a finitely branching hyperarithmetic tree with a hyperarithmetic bound on its branching. Running the Cantor-Bendixon analysis of closed sets within $L_{\wock}$ we see that every element of~$A$ is hyperarithmetic, and so~$f$ is.
\end{proof}

We will show that no higher weakly 2-random set can have a closed approximation. Thus, to separate $\Pi^1_1$-randomness from higher weak 2-randomness we will need to find a class strictly between compact approximations and collapsing approximations. 

\medskip

Narrowing our classes further we return to the idea of counting the number of changes. Finite-change approximations have been implicitly used by Yu \cite{Yu2011}. 

\begin{definition} \label{def:finite_change_approximations}
	A $\wock$-computable approximation $\seq{f_s}$ is a \emph{finite-change} approximation if for no $n$ is there an increasing infinite sequence $\seq{t(i)}_{i<\w}$ of stages such that $f_{t(i+1)}(n)\ne f_{t(i)}(n)$ for all $i<\w$. 
\end{definition}

Note that it is not enough to require that there are only finitely many stages~$s$ such that $f_{s+1}(n)\ne f_s(n)$. For it is possible that there are limit stages~$s$ at which a new value is given. On the other hand, if $\seq{f_s}$ changes only finitely often then for all limit $s$, $\lim_{t\to s} f_t$ exists. Since this limit is $\wock$-computable from~$s$, we may assume that for all limit $s$, $f_s = \lim_{t\to s} f_t$. In this case, we can indeed define the \emph{number of changes} on~$n$ to be the number of stages~$s$ such that $f_{s+1}(n)\ne f_s(n)$. Without this assumption we can define the number of changes to be the longest length of any  increasing sequence $\seq{t(i)}$ of stages such that $f_{t(i+1)}(n)\ne f_{t(i)}(n)$. To pay a debt, we mention the definition of higher $\w$-c.a.\ functions.
\begin{definition} \label{def:omega_c.a.}
	A \emph{higher $\w$-computable approximation} is a finite-change $\wock$-computable approximation $\seq{f_s}$ for which the number of changes is bounded by a hyperarithmetic function.
\end{definition}
Like its lower analogue, a function has a higher $\w$-computable approximation if and only if it is higher truth-table reducible to~$O$.

\medskip

Suppose that $\seq{f_s}$ is a finite-change approximation which has been modified so that $f_s = \lim_{t\to s} f_t$ for all limit ordinals~$s$. Then the set $\{f_s\,:\, s<\wock\}\cup\{f\}$ is a closed subset of Baire space. Further, because this is a finite-change approximation, it is contained in the set of paths of a finitely branching subtree of $\w^{<\w}$, which is compact. Hence:

\begin{lemma} \label{lem:finite-change_implies_compact}
 If~$f$ has a finite-change approximation then it has a compact approximation.
\end{lemma}

The simplest finite-change approximation is an $\wock$-enumeration of a $\Pi^1_1$ set, or a monotone approximation of a higher left-c.e.\ (left-$\Pi^1_1$) real. Chong and Yu showed \cite{ChongYu} that a higher left-c.e.\ sequence cannot be higher weak 2-random. Their proof used the Lebesgue density theorem. \Cref{lem:finite-change_implies_compact,prop:w2r_closed_approx} give a new proof of their result. They also answer Yu's question whether the two halves of higher $\Omega$ are $\Pi^1_1$-random or not. Since they both have a finite-change approximation, they are not even higher weakly 2-random. Indeed this gives us a separation of higher weak 2-randomness from higher difference randomness, since the two halves of higher $\Omega$ do not higher compute each other and so are $\le_{\hT}$-incomplete.

\section{Extremes of higher Turing and higher c.e.}

Before we discuss randomness we investigate the notions of higher relative computability and enumeration, in particular when they coincide with familiar notions. With very strong oracles they collapse to the familiar notions of Turing reducibility and relative computable enumerability. With weak oracles they coincide with relative~$\Delta^1_1$ and~$\Pi^1_1$.

\subsection{Higher computability and strong oracles}

\begin{proposition} \label{prop:O_collapses_higher_computability}
	A set is higher $O$-c.e.\ if and only if it is $O$-c.e.; and so a set is higher $O$-computable if and only if it is $O$-computable. Furthermore, these equivalences hold when~$O$ is replaced by any oracle $Y\ge_{\Tur} O$. 
\end{proposition}

\begin{proof}
	The point is that~$O$ computes a bijection between~$\w$ and~$\wock$, and so relative to~$O$, quantifiers ranging over~$\wock$ can be transformed to quantifiers ranging over~$\w$. Formally, there is an $O$-computable binary relation $E\subset \w^2$ such that $(\w,E)\cong (L_{\wock},\in)$, and further, such that $f\rest{\w}$ is $O$-computable, where $f\colon L_{\wock}\to (\w,E)$ is the unique isomorphism. Every set which is higher $O$-c.e.\ is $\Sigma_1((L_{\wock},\in), O)$-definable, and so, if~$X$ is higher $O$-c.e.\ then $X = f^{-1}Z$ where~$Z$ is $\Sigma_1$-definable in the structure $(\w,E)$ and so~$Z$ (and so~$X$) is $O$-c.e.
\end{proof}

As mentioned in the introduction, there is an effective higher enumeration of all~$\Pi^1_1$ sets, and so we can define an effective higher enumeration of all higher enumeration functionals. We will use the familiar notation $\seq{W_e}$ to denote such an enumeration. We will never use both c.e.\ sets and $\Pi^1_1$ sets in the same context so no confusion should arise. The enumeration gives rise to a higher jump operator $Y\mapsto J^{Y} = \bigoplus_{e<\w} W_e^{Y}$, for which we easily verify $Y <_{\hT} J^{Y}$ for every~$Y$. Since~\cref{prop:O_collapses_higher_computability} is uniform in the indices for~$O$-c.e.\ and higher~$O$-c.e.\ sets, we see that in the particular case where $Y\ge_{\Tur} O$,  the higher jump $J^{Y}$ and the standard Turing jump~$Y'$ are recursively isomorphic. On the other hand, $O$ is recursively isomorphic to~$J^{\emptyset}$.  

\

Reals which have collapsing approximations (\cref{def:collapsing_approximation}) are computationally strong in that they compute a copy of $\wock$. Recall that $X\le_{fin-h}Y$ if $X\le_{\hT}Y$ via a strongly consistent functional: one whose graph is a monotone function from strings to strings, whose domain is closed under taking initial segments. 

\begin{proposition} \label{prop:collapsing_and_fin-h}
	Suppose that $Y\in 2^\w$ has a collapsing approximation. Then for every higher $Y$-computable set~$X$ we actually have $X\le_{fin-h} Y$. 
\end{proposition}

\begin{proof}
	Let $\Phi$ be a higher Turing functional such that $\Phi(Y)=X$, and let $\seq{Y_s}$ be a collapsing approximation for~$Y$. We may assume that for all $\s\in 2^{<\w}$, $|\Phi(\s)|\le |\s|$. We define a fin-h functional $\Psi$ by recursion, by selectively copying~$\Phi$-computations. At stage~$s$ let $\Psi_{s}$ consist of all the axioms already enumerated into~$\Psi$ by stage~$s$. For every $n<\w$, if:
	\begin{itemize}
		\item $Y_s\rest n$ is not in the domain of $\Psi_s$; and
		\item $\Phi_s(Y_s\rest n)$ is consistent,
	\end{itemize}
	then we enumerate an axiom mapping $Y_s\rest n$ to $\Phi_s(Y_s\rest{n})$ into $\Psi_{s+1}$. Then~$\Psi$ is a fin-h functional. It suffices to show that $\Psi(Y)$ is total. Let $n<\w$ and let $s(n)$ be the least $s$ such that $Y_s\rest n = Y\rest n$.  Since the approximation is collapsing, there is some $k\ge n$ such that $\Phi_{s(k)}(Y\rest n) = \Phi(Y\rest n)$. Also, $Y\rest{k}$ is not in the domain of $\Psi_{s(k)}$, and $\Phi(Y\rest{k})[s(k)]$ is consistent and extends $\Phi(Y\rest n)$. It follows that $\Psi_{s(k)+1}(Y)\succeq \Phi(Y\rest n)$. 
\end{proof}

On the other hand we know that there are $O$-computable sets $X$ and $Y$ such that $X\le_{\hT}Y$ but $X\nle_{fin-h} Y$, so some assumption on the nature of the approximations is necessary.

\subsection{Higher computability and relative $\Pi^1_1$} 

\begin{proposition} \label{prop:higher_Turing_when_wock_is_preserved}
	Suppose that~$Y$ preserves $\wock$ (that is, $\w_1^Y = \wock$). Then for all~$X$, $X\le_{\hT}Y$ if and only if $X\le_{\Tur} Y\oplus H$ for some hyperarithmetic set~$H$. 
\end{proposition}

\begin{proof}
	If $H$ is hyperarithmetic and $X\le_{\Tur} Y \oplus H$ then we can easily devise a hyperarithmetic functional $\Phi$ such that $\Phi(Y) = X$, and so $X\le_{\hT}Y$.
	
	In the other direction, suppose that $\Phi$ is a $\Pi^1_1$ functional, $\Phi(Y)=X$ and $\w_1^Y = \wock$. Let $\seq{\Phi_s}_{s<\wock}$ be an effective enumeration of~$\Phi$. Define $f\colon \w\to \wock$ by letting $f(n)$ be the least stage $s<\wock$  such that $\Phi_s(Y)$ extends $X\rest n$. The function~$f$ is $\Delta_1$-definable over $L_{\wock}(Y)$; since~$Y$ preserves $\wock$, $f$ is bounded below~$\wock$. Let $s<\wock$ bound the range of~$f$. Then $\Phi_s(Y)=X$ and so $X\le_{\Tur} Y\oplus \Phi_s$; and $\Phi_s$ is hyperarithmetic.
\end{proof}

For $Y\in 2^\w$ we let $\Delta^1_1\oplus Y$ be the class of sets Turing reducible to $H\oplus Y$ for some hyperarithmetic set~$H$. Thus \cref{prop:higher_Turing_when_wock_is_preserved} says that if~$Y$ preserves~$\wock$ then $\Delta^1_1\oplus Y$ is the class of sets higher Turing reducible to~$Y$. Unfortunately the proposition cannot be reversed. This can be seen by considering the Borel rank of the set of oracles for which $\Delta^1_1\oplus Y$ equals the collection of sets higher Turing reducible to~$Y$, which is fairly low, whereas the Borel rank of the reals which collapse $\wock$ is high (precisely $\boldsymbol{\Sigma^0_{\wock+2}}$ \cite{Steel1978}). Alternatively we can observe that if $Y\ge_{\Tur} O$, then $\Delta^1_1\oplus Y$ is of course the collection of~$Y$-computable sets, which by \cref{prop:O_collapses_higher_computability} equals the collection of sets higher Turing reducible to~$Y$.

\begin{remark} \label{rmk:higher_Omega_and_O}
	Let~$\Omega$ denote the higher version of Chaitin's left-c.e.\ random number. A standard argument shows that $\Omega \equiv_{\hT} O$, indeed the equivalence is higher weak-truth-table. However since higher~$\Omega$ is~$\Delta^1_1$-random it does not (Turing) compute any noncomputable hyperarithmetic set, let alone Kleene's~$O$, nor is~$O$ Turing reducible to $\Omega\oplus H$ for any hyperarithmetic set~$H$. This shows that the conclusion of \cref{prop:higher_Turing_when_wock_is_preserved} fails for the oracle~$\Omega$.
\end{remark}

It is well-known that for sufficiently Cohen generic, sufficiently random and sufficiently Sacks generic (with respect to forcing with hyperarithmetic perfect sets) sets~$Y$, $\Delta^1_1(Y) = \Delta^1_1\oplus Y$; we discuss this shortly. Note that this equality does imply that $Y$ preserves $\wock$; if $\w_1^Y>\wock$ then $Y^{(\wock)}$ is not in $\Delta^1_1\oplus Y$.  Thus, if $\Delta^1_1(Y) = \Delta^1_1\oplus Y$ then for all $X$, $X\le_h Y$ if and only if $X\le_{\hT} Y$. 

\medskip

We require a notion of uniformity for this equality. First we settle some notation.

\begin{notation} \label{not:notations_for_ordinals}
	We sometimes blur the distinction between notations for ordinals and the ordinals they denote: if $\alpha\in O$ then we let $\alpha$ denote also the ordinal $|\alpha|_O$; we let $\alpha+1$ be the notation for the successor of $\alpha$, and so on. For $\alpha\in O^Y$, we let $Y^{(\alpha)} = H_\alpha^Y$ be the iteration of the Turing jump along~$\alpha$.  
\end{notation}

\begin{definition} \label{def:uniform_continuity_of_Delta11}
	Let $Y\in 2^\w$. We say that $\Delta^1_1(Y) = \Delta^1_1\oplus Y$ \emph{uniformly in~$Y$} if there is a Turing functional $\Psi$ and a higher $Y$-partial computable function~$g$ (a function whose graph is higher $Y$-c.e.) such that for all $\alpha\in O^Y$, $g(\alpha)\in O$ and $Y^{(\alpha)} = \Psi(Y,\emptyset^{(g(\alpha))},\alpha)$. 
\end{definition}

Recall that a $Y$-hyperarithmetic index for a set $A\in \Delta^1_1(Y)$ is a pair $(e,\alpha)$ where $\alpha\in O^Y$ and $A = \Phi_e(Y^{(\alpha)})$ (where here $\Phi_e$ is the $e\tth$ (\emph{lower}) Turing functional). Similarly, a $\Delta^1_1\oplus Y$-index for a set~$A$ is a pair $(e,a)$ where~$a$ is a hyperarithmetic index for a set~$H\in \Delta^1_1$ and $A = \Phi_e(H,Y)$. Then $\Delta^1_1(Y) = \Delta^1_1\oplus Y$ uniformly in~$Y$ if there is a higher $Y$-partial computable method of transforming a $Y$-hyperarithmetic index for a set $A\in \Delta^1_1(Y)$ to a $\Delta^1_1\oplus Y$-index for the same set. (The reverse direction is uniform for all oracles.)

\begin{proposition} \label{prop:when_higher_ce_is_Pi11}
The following are equivalent for $Y\in 2^\w$:
\begin{enumerate}
	\item A set is higher $Y$-c.e.\ if and only if it is $\Pi^1_1(Y)$.
	\item $\Delta^1_1(Y) = \Delta^1_1\oplus Y$ uniformly in~$Y$. 
\end{enumerate}
\end{proposition}

\begin{proof}
	Assume~(1). Note that since there are universal $\Pi^1_1(Y)$ and higher $Y$-c.e.\ sets, the equivalence is uniform: there are computable functions translating between $\Pi^1_1(Y)$-indices and higher $Y$-c.e.\ indices. Given this, we see that the proof of \cref{prop:higher_Turing_when_wock_is_preserved} can be performed effectively in~$Y$, as follows. Given $\alpha\in O^Y$ we obtain indices for higher enumeration functionals which with oracle~$Y$ enumerate $A=Y^{(\alpha)}$ and its complement. As a result we obtain an index for a higher Turing functional $\Phi$ such that $A = \Phi(Y)$ (\cref{prop:relative_c.e._and_relative_Turing} is uniform). The relation ``$\Phi_s(Y)$ is total'' is $\Delta_1$-definable over $L_{\wock}(Y)$ (uniformly in $\Phi$ and $s<\wock$); the argument of \cref{prop:higher_Turing_when_wock_is_preserved} gives us a function~$g$ satisfying $Y^{(\alpha)} = \Psi(Y,\emptyset^{g(\alpha)},\alpha)$ which is $\Pi^1_1(Y)$-definable. Applying~(1) again, we see that~$g$ is higher $Y$-partial computable. 
	
	\
	
	Assume~(2), and let~$g$ witness the uniformity. We recall that we can view~$O$ as a subset of~$O^Y$ (as the set of notations in $O^Y$ which hereditarily do not look at the oracle~$Y$ when computing increasing sequences of notations). Uniformly in $\alpha\in O$ we can get a $\Delta^1_1(Y)$-index for 
	\[ O_\alpha^Y = \left\{ \beta\in O^Y  \,:\, \beta<\alpha \right\};\]
and the point is that $O^Y = \bigcup_{\alpha\in O} O_\alpha^Y$, as~$Y$ preserves~$\wock$. Using~$g$ and varying over $\alpha\in O$ we see how to enumerate $O^Y$ in a higher $Y$-c.e.\ fashion.
\end{proof}

\begin{porism} \label{porism:slightly_less_uniform}
	In \cref{prop:when_higher_ce_is_Pi11} we may replace the definition of uniformity of $\Delta^1_1(Y) = \Delta^1_1\oplus Y$ by the apparently weaker condition that $Y^{(\alpha)} = \Psi(Y,\emptyset^{g(\alpha)},\alpha)$ for all $\alpha\in O$ (rather than all $\alpha\in O^Y$). Spector showed (see \cite[II2.4]{Sacks1990}) that there is a Turing functional $\Gamma$ such that for all $\alpha\in O^Y$, $O^Y_\alpha = \Gamma(Y^{(\alpha+1)})$. If $\alpha\in O$ then $\alpha+1\in O$. In the proof of (2)$\Then$(1) we apply~$g$ to $\alpha+1$. 
\end{porism}

\subsection{The behaviour of generics for various forcing notions} 

We discuss \cref{prop:when_higher_ce_is_Pi11} in the context of Cohen genericity, randomness and Sacks genericity. The ideas here are certainly not new, but some are hard to find in print in the form below. 

\subsubsection{Cohen generics}

It is well-known, via the analysis of Cohen forcing, that if~$G$ is Cohen generic then $\Delta^1_1(G) = \Delta^1_1\oplus G$ uniformly. We will employ the following direct definition of the class of $\Sigma^0_\alpha$ sets; see for example \cite{AshKnight2000}.

\begin{definition} \label{def:alternative_hyp_hierarchy}
	For $\alpha\in O$ we define the class of $\Sigma^0_\alpha$ sets (of numbers and of reals) and indices for these sets. For $\alpha=1$, the $\Sigma^0_1$ sets are the c.e.\ sets (and c.e.\ open sets of reals), with $(e,1)$ being the index of the~$e\tth$ such set in some effective listing. 
	
 Let $\alpha>1$. A set is $\Sigma^0_{<\alpha}$ if it is $\Sigma^0_\beta$ for some $\beta<_{O} \alpha$. A $\Sigma^0_{<\alpha}$-index for such a set is a $\Sigma^0_\beta$-index for some $\beta<_O \alpha$. A set~$A$ is $\Sigma^0_\alpha$ if it is the effective union of $\Pi^0_{<\alpha}$ sets. That is, if there is a c.e.\ set~$W$ such that $A$ is the union of the complements of the sets whose $\Sigma^0_{<\alpha}$-indices are in~$W$. The $\Sigma^0_\alpha$-index for this union is $(e,\alpha)$, where~$W$ is the $e\tth$ c.e.\ set. Note that we do not require that all elements of~$W$ are $\Sigma^0_{<\alpha}$ indices, so $(e,\alpha)$ is a $\Sigma^0_\alpha$ code for all~$e<\w$. 
\end{definition}

Let $Y\in 2^\w$. Note that $Y^{(\omega)}$ is not a $\Sigma^0_\omega(Y)$-complete set, as it is only $\Delta^0_\omega(Y)$ (a uniform disjoint union of $\Sigma^0_n(Y)$ sets for $n$ unbounded). For this reason, we use alternative notation (used in \cite{GreenbergMontalbanSlaman2013}, following ideas from \cite{AshKnight2000}) denoting $\Sigma^0_\alpha(Y)$-complete sets. For $n<\w$ let $Y_{(n)} = Y^{(n)}$. For infinite $\alpha \in O$ let $Y_{(\alpha)} = Y^{(\alpha+1)}$. Also if $\alpha$ is limit let $Y_{(\alpha-1)} = Y^{(\alpha)}$, whereas for $\alpha$ successor $Y_{(\alpha-1)}$ keeps its obvious meaning. For all $\alpha \ge 1$, a subset of~$\w$ is $\Sigma^0_\alpha(Y)$ if and only if it is c.e.\ in $Y_{(\alpha-1)}$. For all $\alpha\ge 1$, $Y_{(\alpha)}$ is recursively isomorphic to the set of numbers~$e$ such that~$Y$ belongs to the~$e\tth$ $\Sigma^0_\alpha$ set of reals. The isomorphism is uniform in~$\alpha$.

When discussing open and closed sets we run into an annoying fact: there is an open set~$U$ which is a $\Sigma^0_2$ set of reals, but for which the predicate $[\s]\subseteq U$ is not~$\Sigma^0_2$. 
The fact that such a set is~$\Sigma^0_2$ will not be too helpful for us. For this reason we call an open set \emph{$\Sigma^0_\alpha$-open} if the set of cylinders contained in it is a $\Sigma^0_\alpha$ set of numbers; equivalently, if it is $\Sigma^0_1(\emptyset_{(\alpha-1)})$. The complement of such a set is called \emph{$\Pi^0_\alpha$-closed}. 

\medskip

The following is the effective version of the fact that all Borel sets have the property of Baire. Recall that for an open set~$V$ we let $\partial V = \bar V\setminus V$, the boundary of~$V$, be the set-theoretic difference between the closure of~$V$ and~$V$ itself. 

\begin{proposition} \label{prop:effective_Baire_property}
	Suppose that~$A$ is a $\Sigma^0_\alpha$ set of reals. Then there is a $\Sigma^0_\alpha$-open set~$U$ such that the symmetric difference $A\symdiff U$ is contained in the union $\bigcup \partial V_n$ where each $V_n$ is a $\Sigma^0_{<\alpha}$-open set. Indices for~$U$ and each~$V_n$ can be obtained effectively from an index for~$A$. 
\end{proposition}

\begin{proof}
	If the proposition holds for~$\alpha$ then for every $\Pi^0_\alpha$ set~$B$ there is a $\Sigma^0_{\alpha+1}$-open set~$W$ such that the symmetric difference $B\symdiff W$ is cointained in the union $\bigcup \partial V_n$ where each~$V_n$ is $\Sigma^0_\alpha$-open; if $U$ is the open set given for the complement of~$B$ then $W$ is the complement of the closure of~$U$, and to the list of sets~$V_n$ we add the set~$U$. Once this is known, the proposition follows by induction on $\alpha$, using the fact that $(\bigcup A_n) \symdiff (\bigcup V_n) \subseteq \bigcup (A_n\symdiff V_n)$. 
\end{proof}


	Let $\alpha\in O$. A real~$G\in 2^\w$ is called ${<\alpha}$-Cohen generic if it does not lie on the boundary of any~$\Sigma^0_{<\alpha}$-open set. For example $n$-genericity is $<(n+1)$-genericity and arithmetical genericity is $<\w$-genericity.	\Cref{prop:effective_Baire_property} implies that if~$G$ is ${<\alpha}$-Cohen generic then $G_{(\alpha)}$ is c.e.\ in~$G \oplus \emptyset_{(\alpha-1)}$. This implies that $G_{(\alpha-1)}$ is computable in $G\oplus\emptyset_{(\alpha-1)}$. Unravelling the notation, this means:
\begin{itemize}
	\item If $n<\w$ and~$G$ is $n$-generic, then $G^{(n)}\equiv_\Tur G \oplus \emptyset^{(n)}$;
	\item If $\alpha\ge \w$ and~$G$ is ${<\alpha}$-generic, then $G^{(\alpha)}\equiv_\Tur G \oplus \emptyset^{(\alpha)}$. 
\end{itemize}

The equivalence is uniform in~$\alpha$. 

In \cite{Pi11RandomnessPaper} we show that a $\Delta^1_1$-Cohen generic set preserves~$\wock$ if and only if it is $\Sigma^1_1$-generic. \Cref{porism:slightly_less_uniform} implies that if~$G$ is $\Sigma^1_1$-generic then a set is $\Pi^1_1(G)$ if and only if it is higher~$G$-c.e.

\subsubsection{Random reals}

It is well-known that if~$Z$ is 2-random then~$Z$ is generalised low: $Z' \equiv_{\Tur} Z\oplus\emptyset'$.

The following is the effective version of the fact that all Borel sets are Lebesgue measurable. It is treated in the theses of Kurtz and Kautz (for the arithmetic hierarchy); see \cite[Thm 6.8.3]{DowneyH2010}. 

\begin{proposition} \label{prop:effective_Lebesgue_measurability}
	Let $\alpha\in O$. For any $\Sigma^0_\alpha$ set of reals~$A$ and positive $q\in \Rat$ there are:
	\begin{itemize}
		\item a $\Sigma^0_\alpha$-open set $U\supseteq A$ such that $\leb(U\setminus A) \le q$; and
		\item a $\Pi^0_{<\alpha}$-closed set $F\subseteq A$ such that $\leb(A\setminus F) \le q$.
	\end{itemize}
	An index for~$U$ can be obtained effectively from an index for~$A$ and from~$q$, using the oracle $\emptyset_{(\alpha-1)}$. An index for~$F$ can be obtained effectively from an index for~$A$ and from~$q$, using the oracle $\emptyset_{(\alpha)}$. All calculations are uniform in~$\alpha$. 
\end{proposition}

A real is called $\alpha$-random if it avoids all nested tests $\seq{A_n}$ where $A_n$ are uniformly $\Sigma^0_\alpha$ sets (not necessarily open). We require that $\leb(A_n)\le 2^{-n}$. \Cref{prop:effective_Lebesgue_measurability} implies that a real is $\alpha$-random if and only if it is ML-random relative to $\emptyset_{(\alpha-1)}$. Uniformly in~$\alpha$ we have a universal ML-test $\seq{U^\alpha_n}$ relative to $\emptyset_{(\alpha-1)}$. An \emph{$\alpha$-randomness deficiency} of an $\alpha$-random real~$Z$ is some~$n$ such that $Z\notin U^\alpha_n$. If $\seq{V_n}$ is any ML-test relative to $\emptyset_{(\alpha-1)}$ (so the sets $V_n$ are uniformly $\Sigma^0_\alpha$-open) then from an~$\alpha$-randomness deficiency of an~$\alpha$-random real~$Z$ and an index for the sequence $\seq{V_n}$ we can effectively find some~$m$ such that $Z\notin V_m$. If $\beta<_O \alpha$ and $Z$ is $\alpha$-random then of course it is also $\beta$-random, and a $\beta$-randomness deficiency of~$Z$ can be effectively found from an $\alpha$-randomness deficiency of~$Z$.

Chong and Yu \cite{ChongYu} observed that $\Delta^1_1(Z) = \Delta^1_1\oplus Z$ uniformly for any $\Delta^1_1$-random real~$Z$ which preserves $\wock$. We prove a more precise version of this result. 

\begin{proposition} \label{prop:generalised_lowness_for_randoms}
	Let $\alpha\ge 2$. If $Z$ is $\alpha$-random then $Z_{(\alpha-1)}\le_\Tur Z\oplus \emptyset_{(\alpha-1)}$. An index for the reduction can be found effectively from an $\alpha$-randomness deficiency of~$Z$. This is uniform in~$\alpha$. 
\end{proposition}

In short, for all $\alpha\ge 1$, if $Z$ is ML-random relative to $\emptyset^{(\alpha)}$ then $Z^{(\alpha)}\equiv_\Tur Z \oplus \emptyset^{(\alpha)}$. Note the difference at infinite levels compared with Cohen genericity. For example, if~$G$ is arithmetically Cohen generic then $G^{(\omega)}\equiv_{\Tur}G\oplus \emptyset^{(\omega)}$. In contrast, by forcing with arithmetical sets with positive measure one obtains an arithmetically random set~$Z$ for which the equation fails. 

\begin{proof}
We show this in two steps. First we consider successor ordinals~$\alpha$. Suppose that $\alpha = \beta+1$. We need to show that if~$Z$ is $\beta+1$-random (ML random relative to $\emptyset_{(\beta)}$) then $Z_{(\beta)}\le_\Tur Z\oplus\emptyset_{(\beta)}$. Given a $\Sigma^0_\beta$ set of reals~$A$ we want to decide whether $Z\in A$ or not. Using $\emptyset_{(\beta)}$ we find sequences $\seq{U_n}$ and $\seq{F_n}$ such that $U_n$ is $\Sigma^0_\beta$-open, $F_n$ is $\Pi^0_{<\beta}$-closed, $F_n \subseteq A \subseteq U_n$ and $\leb (U_n\setminus F_n)\le 2^{-n}$. The sequence $\seq{U_n\setminus F_n}$ is a $\emptyset_{(\beta)}$-ML test, and so we can find some~$n$ such that $Z\notin (U_n\setminus F_n)$. Thus $Z\in A$ if and only if $Z\in F_n$. To determine whether $Z\in F_n$ we employ a similar process. $F_n$ is $\Pi^0_\gamma$-closed for some $\gamma<\beta$. Relativising the case $\alpha=1$ to $\emptyset_{(\gamma)}$ we obtain a $\emptyset_{(\gamma)}$-computable sequence $\seq{C_m}$ of clopen supersets of $F_n$ such that $\leb(C_m\setminus F_n)\le 2^{-m}$. Again this is a $\emptyset_{(\gamma)}$-test and so we can find some $m$ such that $Z\notin (C_m\setminus F_n)$. We conclude that $Z\in A$ if and only if $Z\in C_m$, and this can of course be checked directly with the oracle~$Z$. 

Next we consider limit ordinals~$\alpha$. If $Z$ is $\alpha$-random (ML-random relative to~$\emptyset^{(\alpha)}$) then uniformly in $\gamma<_O \alpha$ it is $\gamma$-random (by this we mean that we can, uniformly in $\gamma$, compute an upper bound on the $\gamma$-randomness deficiency of~$Z$). As $Z_{(\alpha-1)} = Z^{(\alpha)}$ is the effective join $\bigoplus_{\gamma<_O \alpha} Z^{(\gamma)}$, to compute $Z^{(\alpha)}$ it suffices to compute each $Z^{(\gamma)}$, and we may restrict ourselves to successor ordinals~$\gamma$. However with oracle $\emptyset^{(\alpha)}$ we uniformly obtain $\emptyset^{(\gamma)}$ and we have already shown that $Z^{(\gamma)}\le_{\Tur} Z\oplus \emptyset^{(\gamma)}$ uniformly.
\end{proof}

\begin{remark} \label{rmk:weakly_not_enough}
	The components of the $\emptyset_{(\beta)}$-ML tests described in the proof of \cref{prop:generalised_lowness_for_randoms} are all $\Sigma^0_{\beta}$ rather than $\Sigma^0_{\beta+1}$. These are equivalent to weak $\beta$-tests (generalized $\emptyset_{(\beta-1)}$-ML tests). It would seem that we could relax the randomness requirement. However the key is the uniformity in~$A$: for each~$A$ we have a different test, and the full $\beta+1$-randomness deficiency of~$Z$ is used to find components of these tests that~$Z$ avoids. Indeed, Lewis, Montalb\'an and Nies \cite{LewisMontalbanNies2007} showed that there is a weakly 2-random set which is not generalized low. 	
\end{remark}

Stern \cite{Stern1975} and independently Chong, Nies and Yu~\cite{ChongNiesYu2008} showed that a $\Delta^1_1$-random real is $\Pi^1_1$-random if and only if it preserves~$\wock$. Suppose that~$Z$ is $\Pi^1_1$-random. Then it is $\Pi^1_1$-ML random (higher ML-random). From a hyperarithmetic index for a ML-test relative to some hyperarithmetic oracle we can effectively find an index for this test as a sequence of uniformly~$\Pi^1_1$ open sets. Hence from a randomness deficiency for~$Z$ as a higher ML-random real we can uniformly in $\alpha\in O$ find an $\alpha$-randomness deficiency for~$Z$. Consequently, $\Delta^1_1(Z) = \Delta^1_1\oplus Z$ uniformly. Hence, if~$Z$ is $\Pi^1_1$-random, then a set is $\Pi^1_1(Z)$ if and only if it is higher~$Z$-c.e.

\begin{remark} \label{rmk:Chong_Yu_and_Demuth}
	Chong and Yu~\cite{ChongYu} proved an analogue of Demuth's theorem: If~$X$ is $\Pi^1_1$-random, $Y\le_h X$ and $Y$ is not hyperarithmetic, then $\deg_h(Y)$ contains a $\Pi^1_1$-random sequence. The structure of their argument follows that of Demuth's theorem; this can be further clarified using Higher Turing reducibility. In the first step we already know that $Y \le_\Tur X\oplus H$ for some hyperarithmetic set~$H$. Further, being $\Pi^1_1$-random, $X$ is $\Delta^1_1$-dominated: every~$\Delta^1_1(X)$ function is bounded by a hyperarithmetic one. Applying this to the use of the reduction, we see that~$Y$ is higher truth-table reducible to~$X$. This implies that~$Y$ is higher ML-random for the image measure, and since it is not hyperarithmetic, it is not an atom of this measure. The second step of the proof is now identical to the classical one: if~$Y$ is higher ML-random for some hyperarithmetic measure and is not an atom of this measure, then the higher Turing degree of~$Y$ contains a higher ML-random sequence. Being~$\Pi^1_1$-random, $X$ preserves~$\wock$, and so~$Y$ preserves~$\wock$ as well, which implies that any higher ML-random sequence in $\deg_{\hT}(Y)$ is in fact $\Pi^1_1$-random.
\end{remark}

\subsubsection{Sacks generics}

We consider sets which are generic for forcing with perfect hyperarithmetic closed sets. Sacks (see \cite[IV.5]{Sacks1990}) showed that if~$G$ is sufficiently generic for this notion of forcing then~$G$ preserves $\wock$ and has minimal hyperdegree. The proof shows that $\Delta^1_1(G) = \Delta^1_1\oplus G$. However, this is not uniform. We thank Adam Day for pointing this out. 

\begin{proposition} \label{prop:Sacks_generics}
	If~$G$ is sufficiently generic for hyperarithmetic Sacks forcing, then $O^G$ is not higher $G$-c.e.
\end{proposition}

\begin{proof}
In fact we prove more: we prove that, given a countable collection of enumeration functionals $\seq{\Gamma_i}$ (with no assumption on their effectivity), if~$G$ is generic enough, then $\Gamma_i^G \ne O^G$ for all~$i$. 
Consider a given perfect hyperarithmetic closed set, represented by a perfect tree~$T$ and an enumeration functional~$\Gamma$. It is easy to construct a hyperarithmetic set of nodes~$D\subseteq T$, open in~$T$, which is dense in~$T$ but such that the (hyperarithmetic) tree~$T\setminus D$ is perfect. 
Since~$D$ is hyperarithmetic, there exists an~$n$ such that for every real~$X$, $X$ has a prefix in~$D$ if and only if $n \notin O^X$.  If there are no paths~$X$ in~$T$ such that $n \in \Gamma^X$, then the tree~$T\setminus D$, which refines $T$, forces that $\Gamma^G \ne O^G$. Otherwise there is some $\s\in D$ such that $n\in \Gamma^\s$. Then the ``full subtree'' $T_\s$ of nodes in~$T$ comparable with~$\s$ forces that $\Gamma^G \ne O^G$. 
\end{proof}

\section{Continuity and Randomness}

As discussed in the introduction, when trying to establish analogues of familiar theorems of algorithmic randomness, we sometimes need to work around the usage of time tricks. As a first example we consider van-Lambalgen's theorem. The proof of one direction: if $X$ is higher ML-random, and $Y$ is higher $X$-ML-random, then $X\oplus Y$ is higher ML-random --- is identical to the analogous ``lower'' proof. The other direction usually uses a uniform universal ML test, and as discussed in the inttroduction, no such uniform universal test exists in the higher setting. Given an enumeration operator~$U$, we cannot transform every~$U^X$ to an open set with some fixed measure bound. But we show that we can do this for \emph{most} oracles~$X$, and then argue that this suffices. 

In the following lemma and below we think of operators enumrating open sets given oracles as open subsets of the plane; if $U\subset (2^\w)^2$ is open then $U^X$ is the $X$-section of~$U$.

\begin{lemma} \label{lem:oracle_open_set_trimming}
Let~$U\subseteq (2^\w)^2$ be higher effectively open. For every $\epsilon>0$ there is a higher effectively open set $V\subseteq (2^\w)^2$ such that:

\begin{enumerate}
\item If $\leb(U^X) \leq \epsilon$ then $U^X = V^X$; and
\item For all but a set of measure $\epsilon$-many oracles~$X$, $\leb(V^X)\le \epsilon$.
\end{enumerate}
An index for~$V$ can be obtained uniformly from~$\epsilon$ and an index for~$U$. 
\end{lemma}

For the proof we use the \emph{projectum function} $p\colon \wock\to \w$: this is a $\wock$-computable injective function.

\begin{proof}
We enumerate~$V$. For $s<\wock$ we let~$V_s$ be the open set enuemrated by stage~$s$. Suppose that we see the cylinder~$[\s,\tau]$ enumerated into~$U_{s+1}$. Let~$P_s$ be the set of~$X\in [\s]$ such that $\leb(V_s^X\cup [\tau]) > \epsilon$. We find a clopen set~$C_s\subseteq [\s]$ which is close to the complement $[\s]\setminus P_s$ of~$P_s$ inside~$[\s]$: 
\begin{itemize}
	\item $C_s\cup P_s = [\s]$; and
	\item $\leb(C_s\cap P_s) \le \epsilon \cdot 2^{-p(s)}$. 
\end{itemize}
We then let $V_{s+1} = V_s \cup (C_s\times [\tau])$.

We have $V\subseteq U$, and the desired property~(1) holds. To see~(2), let $B = \left\{ X\in 2^\w \,:\,  \leb(V^X)> \epsilon \right\}$. We claim that $B\subseteq \bigcup_{s<\wock} (C_s\cap P_s)$.  Let~$X\in B$. For limit ordinals $s\le \wock$, $V_s = \bigcup_{t<s}V_t$ (here we let $V_{\wock} = V$) and so there is some $s<\wock$ such that $\leb(V_s^X)\le \epsilon$ but $\leb(V_{s+1}^X)> \epsilon$. But then $X\in C_s\cap P_s$. 
Now 
\[
	\leb(B) \le \leb\left( \bigcup P_s\cap C_s \right) \le \epsilon \cdot \sum_{s<\wock} 2^{-p(s)} \le \epsilon
\]	
as~$p$ is injective. 
\end{proof}

\begin{remark} \label{rmk:the_index_wock}
	We apply the notational convention used in the previous proof throughout this paper. If~$X$ is any object which is approximated or enumerated in~$\wock$ many steps then we let $X_{\wock} = X$. For example if~$U$ is a c.e.\ open set and $\seq{U_s}_{s<\wock}$ is a $\wock$-effective enumeration of $U$ then we write $U_{\wock}$ for~$U$; if $\seq{f_s}$ is a $\wock$-computable approximation of a function~$f$ then we let $f_{\wock} = f$. 
\end{remark}

We can now prove a the higher version of van Lambalgen's theorem.

\begin{proof}[Proof of \cref{thm:van_Lambalgen}]
As disucssed above, the proof of one direction has no new ingredients, and so we omit it. In the other direction we are given a pair $(X,Y)$ and assume that~$Y$ is not higher~$X$-ML random; and need to show that the pair~$(X,Y)$ is not higher ML-random.

Let $\seq{U_n^X}$ be a higher $X$-ML-test which captures~$Y$. By~\cref{lem:oracle_open_set_trimming} we may assume that for all~$n$, the measure of $B_n = \left\{ Z\in 2^\w \,:\,  \leb(U_n^Z) > 2^{-n} \right\}$ is at most $2^{-n}$; $X$ is not in any~$B_n$ and~$Y$ is captured by the~$X$-test after applying the transformation of that lemma. So $(X,Y)\in \bigcap_n U_n$. A calculation (essentially Fubini's theorem) shows that $\leb(U_n) \le (1-2^{-n})\cdot 2^{-n} + 2^{-n}$ which converges to~$0$ (computably). 

\end{proof}

\subsection{Pulling back strong tests} \label{subsec:pulling_back}

The argument of Miller and Yu's, sketched in the introdution, relies on the consistency of the given functional. Recall that a \emph{continuous semi-measure} is a function~$m$ which assigns to every finite binary string a non-negative real number, such that for all $\s\in 2^{<\w}$, $m(\s\conc 0) + m(\s\conc 1)\le m(\s)$. A continuous semi-measure is \emph{higher c.e.} if the real~$m(\s)$ is higher left-c.e., uniformly in~$\s$.  If~$\Phi$ is a consistent functional then the function $\s\mapsto \leb\left(\Phi^{-1}[\s]\right)$ is a continuous semi-measure. In the higher setting not all functionals can be made continuous. However as above, given a functional~$\Psi$ and some $\epsilon>0$ we can transform~$\Psi$ to a functional~$\Phi$ such that $\Phi(X) = \Psi(X)$ if $\Psi(X)$ is consistent, and such that $\Phi(X)$ is inconsistent for at most~$\epsilon$-many (in the sense of measure) oracles. In fact we can combine all the $\epsilon$-modifications in one to get the following. 

\begin{lemma} \label{lem:massage}
	For every higher Turing functional~$\Psi$ there is a higher Turing functional~$\Phi$ such that:
	\begin{enumerate}
		\item for all~$X$ for which $\Psi(X)$ is consistent, $\Phi(X)= \Psi(X)$; and
		\item the function $\tau\mapsto \leb(\Phi^{-1}[\tau])$ is bounded by a higher c.e.\ continuous semi-measure.
	\end{enumerate}
\end{lemma}

\begin{proof}
Fix a function $q \colon 2^{<\w} \to \QQ^+$ such that $\sum_{\tau \in 2^{<\omega}} q(\tau) \le 1$ and such that $\tau_1 \preceq \tau_2$ implies $q(\tau_1) \geq q(\tau_2)$ (for example let $q(\tau) = 2^{-3|\tau|}$). We enumerate a functional~$\Phi$. 

Suppose that we see the axiom $(\s_s,\tau_s)$ enumerated into~$\Psi_{s+1}$. We let~$P_s$ be the set of $X\in [\s_s]$ such that $\Phi_s(X)$ is inconsistent with~$\tau_s$. Let~$C_s$ be a clopen subset of~$[\s_s]$ close to the complement $[\s_s]\setminus P_s$; we mean that $P_s\cup C_s  = [\s_s]$ and $\leb(P_s\cap C_s)\le 2^{-p(s)}\cdot q(\tau_s)$, where as above~$p$ is the projection function. We then declare that $\Phi_{s+1}(X)\succeq \tau_s$ for all $X\in C_s$.

\smallskip

Inductively, for all~$s$ and~$X$, $\Phi_s(X)\preceq \Psi_s(X)$, and so if $X\in P_s$ then $\Psi(X)$ is inconsistent. This establishes~(1). 

For~(2) we let 
\[
	m(\tau) = \leb \left(\Phi^{-1}[\tau]\right) + \sum_{\rho \succeq \tau} q(\rho).
\]
For~$\tau\in 2^{<\w}$ let $B(\tau) = \Phi^{-1}[\tau\conc 0]\cap \Phi^{-1}[\tau\conc 1]$. So
 \[ \leb(\Phi^{-1}[\tau\conc 0])+ \leb(\Phi^{-1}[\tau\conc 1]) \le \leb(\Phi^{-1}[\tau])+ \leb(B(\tau)). \]
If $X\in B(\tau)$ then there is some stage~$s<\wock$ such that $X\in P_s\cap C_s$ and $\tau_s$ extends either $\tau\conc 0$ or $\tau\conc 1$. Since $q(\tau)\ge q(\tau_s)$, the argument of \cref{lem:oracle_open_set_trimming} shows that $\leb(B(\tau))\le q(\tau)$. A calculation now shows that~$m$ is a continuous semi-measure. 
\end{proof}

\Cref{lem:massage} allows us to show that strong randomness notions are downwards closed in the $\hT$-degrees of higher ML-random sets. In particular we get \cref{thm:downward_closure_Strong-ML-randomness}.

\begin{theorem}
Suppose that~$X$ and~$Y$ are higher ML-random and that $X\le_{\hT} Y$. If~$Y$ is higher weakly-2-random (higher difference random, higher $Z$-ML-random for some~$Z\in 2^\w$,\textellipsis{}) then so is~$X$. 
\end{theorem}

\begin{proof}
By \cref{lem:massage} we get a higher Turing functional $\Phi$ such that $\Phi(Y)=X$ and $\leb(\Phi^{-1}[\tau])\le m(\tau)$ for some higher c.e., continuous semi-measure. Since~$X$ is higher ML-random, $m(X\rest n)\le c\cdot 2^{-n}$ for some constant~$c$. 
We can then eumerate a functional $\Psi\subseteq \Phi$ such that $\Psi(Y)=X$ and $\leb(\Psi^{-1}[\tau])\le c\cdot 2^{-|\tau|}$ for all~$\tau$: we enumerate~$\Psi$. At stage~$s$ say an axiom $(\s,\tau)$ appears in~$\Phi_{s+1}$. If $\leb([\s]\cup \Psi^{-1}_s[\rho]) > c\cdot 2^{-|\rho|}$ for some $\rho\preceq \tau$ then we let $\Psi_{s+1}= \Psi_s$; otherwise we let $\Psi_{s+1}= \Psi_s\cup \{(\s,\tau)\}$. In the first case $\leb(\Phi^{-1}[\rho])> c\cdot 2^{-|\rho|}$ and so~$\rho$ is not an initial segment of~$X$; so~$\s$ is not an initial segment of~$Y$. 

If $\seq{U_n}$ is any strong test capturing~$X$ then $\seq{\Phi^{-1}[U_n]}$ is a strong test capturing~$Y$. The point is that $\leb(\Phi^{-1}[U_n]) \le c\cdot \leb(U_n)$. There may not be any higher c.e.\ (higher $Z$-c.e.) antichain generating~$U_n$; but for the measure calculation we do not need effectiveness: the inequality is obtained by considering the antichain of minimal strings (maximal intervals) in~$U_n$. 
\end{proof}

\section{ $K$-triviality}

Hjorth and Nies defined in \cite{HjorthNies2007} the notion of higher prefix-free Kolmogorov complexity, based on the concept of universal $\Pi^1_1$ prefix-free machine. We denote this complexity function by~$K$, as we will not be using the traditonal ``lower'' complexity. Armed with this concept Hjorth and Nies defined the class of higher $K$-trivial sets, those sets~$A\in 2^\w$ satisfying $K(A\rest{n})\le^+ K(n)$. 

Hjorth and Nies proved that there are higher $K$-trivial sets which are not hyperarithmetic (arguing that Solovay's proof applies in the higher setting) and also that every higher $K$-trivial is Turing reducible to Kleene's~$O$. As described in the introduction, since they use discontinuous relativisations, their notions of higher lowness for~$K$, higher bases for randomness and higher lowness for MLR coincide with being hyperarithmetic. Continuous relativisations yield analogues of familiar equivalences. 

\smallskip

	In addition to \cref{thm:higher_K-trivial_coincidence}, we also show that a set is higher $K$-trivial if and only if it is higher low for~$K$. As mentioned above, defining the notion is not completely sraightforward because there are oracles~$A$ for which there is no optimal prefix-free complexity; so~$K^A$ is not well-defined for all~$A$. Further complication is due to the potential failure of the equivalence between prefix-free complexity and discrete c.e.\ measures. Recall that a \emph{discrete measure} (often called a discrete semi-measure, but it is a measure) is simply a measure on~$\w$ (equivalently, on any computable set); such a measure is of course determined by the measures of its atoms. A discrete measure~$\mu$ is called (higher) c.e.\ if $\mu(n)$ is a (higher) left-c.e.\ real, uniformly in~$n$. Nies and Hjorth showed that the higher analogue of the Kraft-Chaitin theorem holds, from which follows the higher analogue of the coding theorem, which says that every higher c.e.\ discrete measure can be realised as the measure dervied from a higher prefix-free machine ($\mu_M = 2^{-K_M}$). Thus $2^{-K}$ is an optimal higher c.e.\ discrete measure. 

	We do not know whether the coding theorem can be continuously relativised to every oracle. Thus given an oracle~$A$ we can investigate both higher $A$-computable prefix-free machines (their graphs are higher $A$-c.e.) and their associated complexities; and higher $A$-c.e.\ discrete measures. This gives two definitions of lowness:
	\begin{itemize}
		\item an oracle~$A$ is \emph{low for higher~$K$} if for every higher $A$-computable prefix-free machine~$M$, $K\le^+ K_M$;
		\item an oracle~$A$ is \emph{low for higher c.e.\ discrete measures} if for every higher $A$-c.e.\ discrete measure~$\nu$, $\+\mu \ge^\times \nu$ where $\+\mu$ is the optimal higher c.e.\ discrete measure. 
	\end{itemize}
	A-priori the second notion is stronger. We will show that both of these concepts coincides with higher $K$-triviality. On the other hand, since the concept of $K$-triviality itself does not involve relativisation, it can be characterised using discrete measures: a set~$A$ is $K$-trivial if and only if $\+\mu(A\rest n) \ge^\times \+\mu(n)$.

\subsection{Approximations of $K$-trivial sets}

The following is implicit in \cite{HjorthNies2007}.

\begin{proposition} \label{prop:K_trivials_are_collapsing}
	Every nonhyperarithmetic higher $K$-trivial set has a collapsing approximation. 
\end{proposition}

In fact if~$A$ is higher $K$-trivial then there is an increasing approximation $\seq{\+\mu_s}$ of~$\+\mu$ and a collapsing approximation $\seq{A_s}$ of~$A$ such that for some constant~$\delta>0$, $\+\mu_s(A\rest n) \ge \delta\cdot \+\mu_s(n)$ for all $n<\w$ and all $s<\wock$. 

\begin{proof}
	We start with an arbitrary enumeration $\seq{\UU_s}$ of the universal higher-c.e.\ prefix-free machine~$\UU$, and let $K_s = K_{\UU_s}$. As usual we assume that the enumeration of~$\UU$ is continuous, i.e. $\UU_s = \bigcup_{t<s} \UU_t$ for every limit ordinal $s\le \wock$. Hence $K_s = \lim_{t\to s} K_t$ for every limit ordinal~$s$. 

	There is a $\wock$-computable sequence of trees~$\seq{T_s}_{s\le\wock}$ such that:
	\begin{itemize}
		\item For all limit $s\le \wock$, $T_s = \lim_{t\to s} T_t$; and
		\item $A$ is the unique path of~$T_{\wock}$. 
	\end{itemize}
	For let~$b$ be a $K$-triviality constant for~$A$. There are only finitely many $K$-trivial sequences with constant~$b$. For $s\le \wock$ let $S_s$ be the tree of finite binary strings which are $K_s$-trivial with constant~$b$. Let~$\s$ be a string on $S_{\wock}$ which isolates~$A$ on $S_{\wock}$. We let $T_s$ be the restriction of~$S_s$ to strings comparable with~$\s$. 

	In~\cite{HjorthNies2007}, Hjorth and Nies show that there is a $\wock$-computable closed and unbounded set $C\subseteq \wock$ such that for all $s\in C$, the tree $T_s$ has only finitely many paths. A similar argument shows that after thinning to a possibly smaller set of stages we may assume that for all $s\in C$, $T_s$ has a path (for all $n$, if $T_t$ contains a string of length~$n$ for all~$t$ in some set~$B$ of stages, then by continuity $T_{\sup B}$ also contains a string of length~$n$.) We define the approximation $\seq{A_s}$ for $s\in C$ by letting~$A_s$ be the leftmost path in~$T_s$. Then $A = \lim_{s\in C} A_s$. This approximation is collapsing: if $A\rest n\in T_{s(n)}$ and $s(\w) = \sup_n s(n)$ then~$A$ is a path in $T_{s(\w)}$; if $s(\w)< \wock$ then $T_{s(\w)}$ is hyperarithmetic, and so each of its finitely many paths is hyperarithmetic.  

	Finally we renumber our approximations using the increasing $\wock$-computable enumeration of~$C$, and let $\+\mu_s = 2^{-K_s}$.
\end{proof}

The fact that a set~$A$ has a collapsing approximation allows us to relativise to~$A$ many familiar techniques, with arguments along the lines of that of \cref{prop:collapsing_and_fin-h}. In the language of \cite{BadOracles}, it is a ``good oracle''. For example:

\begin{lemma} \label{lem:collapsing_implies_Schnorr-Levin}
	Suppose that~$A$ has a collapsing approximation. Then there is an optimal higher~$A$-c.e.\ discrete measure $\+\mu^A$, and a sequence~$Z$ is higher $A$-ML-random if and only if $\+\mu^A(Z\rest n)\le^\times 2^{-n}$. Further, there is a universal higher $A$-c.e.\ prefix-free machine~$\UU^A$ and $\+\mu^A =^\times 2^{-K^A}$. 
\end{lemma}

\begin{proof}
	To get a universal higher $A$-c.e.\ prefix-free machine we show that we can uniformly transform a given enumeration functional~$W$ to an enumeration functional~$V$ such that~$V^A$ is the graph of a function with prefix-free domain (indeed this is true for every oracle), and if $W^A$ is a graph of such a function then $V^A = W^A$. As in the proof of \cref{prop:collapsing_and_fin-h}, if we see that $\s = A_s\rest n$ is not an initial segment of~$A_t$ for any $t<s$, and $W_s^\tau$ is the graph of a function with prefix-free domain, then we let $V_s^\tau = W_s^\tau$. 

	In the same way we get $\+\mu^A$; if $\mu\subseteq 2^{<\w}\times \w\times \Rat^+$ then we let, for each $n<\w$ and $X\in 2^{\le \w}$, $\mu^X(n)= \sup \left\{ q\in \Rat^+ \,:\,  (\s,n,q)\in \mu\text{ for some }\s\preceq X \right\}$; and let $\mu^X(\w) = \sum_{n\in \w} \mu^X(n)$. We can transform each higher c.e.\ such~$\mu$ into some~$\nu$ such that $\nu^A(\w)\le 1$ and if $\mu^A(\w)\le 1$ then $\nu^A=\mu^A$: when we see a ``fresh'' $\tau\prec A_s$, we copy~$\mu_s^\tau$, provided that $\mu_s^\tau(\w)\le 1$. 

	\smallskip

	The key step in the standard (``lower'') proof of the Levin-Schnorr theorem (the equivalence of discrete measures and tests in capturing ML-randomness) is taking an effectively open set~$U$ and obtaining a c.e.\ prefix-free set generating~$U$. In the higher setting this is impossible; using the projectum funcion and approximations of closed sets from above by clopen sets, we can get a set of strings generating~$U$ whose \emph{weight} is bounded by $\leb(U)+\epsilon$ for any prescribed~$\epsilon>0$. However working relative to an oracle~$A$ with a collapsing approximation makes the situation \emph{easier}: in some sense the collapsing approximation brings us closer to~$\w$-computability. If~$A$ has a collapsing approximation and~$U^A$ is higher $A$-effectively open then there is a higher $A$-c.e.\ prefix-free set of strings~$W^A$ generating~$U$: if $\tau\prec A_s$ is fresh then we enumerate into~$W_{s+1}^\tau$ all strings~$\s$ \emph{of length~$|\tau|$} such that $[\s]\subseteq U^A_s$ but~$[\s]$ is disjoint from $[W_s^\tau]$. 

	In a similar way, relative to~$A$ we can follow the standard proof of the Kraft-Chaitin / coding theorem without having to resort to the necessary complications of the proof of the unrelativised theorem in the higher setting (see \cite{HjorthNies2007}).\footnote{Another way to understand the situation is to observe that a collapsing approximation of~$A$ gives us an $\w_1^{ck}$-$A$-computable $\w$-sequence $\seq{\alpha_n}$ cofinal in~$\wock$. From this we get a relation $E \le_{\hT} A$ such that $(\w,E)\cong (L_{\wock},\in)$ (and as is the situation with~$O$, we can make the map $n\mapsto n^{(\w,E)}$ computable). This means that the higher $A$-c.e.\ sets are precisely those which are~$\Sigma_1$-definable in the structure $(L_{\wock},\in,A)$. So when designing higher $A$-c.e.\ sets we don't have to consider other oracles, as is usually the case with desining oracle-c.e.\ sets; and we can enumerate such sets using a recursion of length~$\w$ along the sequence $\seq{\alpha_n}$. All familiar constructions can be performed this way. For example when enumerating a higher~$A$-effectively open set~$U$ we may assume that by stage~$\alpha_n$, only strings of length~$n$ have been enumerated into~$U$.}
\end{proof}

Suppose that~$A$ is low for higher~$K$. Then it is higher~$K$-trivial. With \cref{lem:collapsing_implies_Schnorr-Levin} we can then conclude that it is also low for higher c.e., discrete measures, and low for higher ML-randomness.

\subsection{Hungry sets} 
\label{sub:hungry_sets}

We next show that if $A$ is a base for higher randomness ($A\le_{\hT}Z$ for some higher $A$-ML-random set~$Z$) then~$A$ is higher~$K$-trivial. 
\begin{itemize}
	\item We could modify the argument to obtain lowness for higher~$K$. We will later show though that higher $K$-triviality implies lowness for higher~$K$. 
	\item The higher version of the Ku\v{c}era-G\'{a}cs theorem shows that if~$A$ is low for higher ML-randomness then it is a base for higher randomness. So we also conclude that lowness for higher ML-randoness implies higher $K$-triviality and therefore lowness for higher~$K$. A more direct argument is likely possible but for brevity we omit it. 
\end{itemize}

We need to carry out the ``hungry sets'' construction of \cite{HirschfeldtNiesStephan2007}. In~\cite{HjorthNies2007} the authors claim that the proof carries over with only notational changes; they ignore the typical topological problems. These problems are present even if one assumes that the reduction of~$A$ to~$Z$ is a fin-h reduction; the problems increase slightly when inconsistent functionals are admitted. Here we discuss these problems and show how to overcome them. 

\medskip

We recall the structure of the proof. Suppose that $\Phi(Z)=A$ where~$\Phi$ is a higher Turing functional and~$Z$ is higher $A$-ML-random. We fix $\epsilon>0$. We enumerate ``hungry sets'' $C^\alpha = C^\alpha(\epsilon)$ for every finite binary string~$\alpha$; we ensure that $C^\alpha\subseteq \Phi^{-1}[\alpha]$. An attempt to show that~$A$ is $K$-trivial is made by ensuring that $\alpha\mapsto \leb(C^\alpha)$ is a higher-c.e.\ discrete measure, and attempting to show that~$\leb(C^{A\rest n})\ge^\times \+\mu(n)$. So we aim to ensure three things:
\begin{enumerate}
	\item the measure of $\bigcup_{\alpha\prec A} C^\alpha$ is bounded by~$\epsilon$; 
	\item either for all $\alpha\prec A$, $\leb(C^\alpha)  =\epsilon \+\mu(|\alpha|)$, or  $Z\in \bigcup_{\alpha\prec A} C^\alpha$; and
	\item the sum $\sum_{\alpha\in 2^{<\w}} \leb(C^\alpha)$ is finite. 
\end{enumerate}

We ensure that for all~$s$ and~$\alpha$, $\leb(C^\alpha_s)\le \epsilon \+\mu_s(|\alpha|)$; this ensures~(1). In the standard proof, (3) is obtained by ensuring that the hungry sets are pairwise disjoint. The usual topological reasons preculde this from hapenning in the higher setting; at an infinite stage~$s$, $\Phi_s^{-1}[\alpha]\setminus C_s^\alpha$ may have positive measure but no interior. Further, if~$\Phi$ is inconsistent then we do not automatically get that~$C^\alpha$ and~$C^\beta$ are disjoint if~$\alpha$ and~$\beta$ are incomparable. As above, we remedy this by allowing overlap, but ensuring that it is small.

Fix positive rational numbers~$\delta_\alpha$ for all strings $\alpha\in 2^{<\w}$, so that $\sum_{\alpha\in 2^{<\w}} \delta_\alpha$ is finite. For notational simplicity at each stage of the construction we consider a single string~$\alpha$ (at stage~$t+n$, $t$ limit, consider the $n\tth$ finite binary string). Let $C_s = \bigcup_{\beta\in 2^{<\w}} C^\beta_s$. We find a clopen $B_s\subseteq C_s$ such that $\leb(C_s\setminus B_s)\le \delta_\alpha 2^{-p(s)}$. We now consider:
\begin{itemize}
	\item $G_s = \Phi^{-1}_s[\alpha]\setminus B_s$ --- this is potential fodder;
	\item $q_s = \epsilon \mu_s(|\alpha|) - \leb(C^\alpha_s)$ --- this is the amount we would like to add to~$C^\alpha$. 
\end{itemize}
If $\leb(G_s)\le q_s$ then we let $C^{\alpha}_{s+1} = C^\alpha_s\cup G_s$. If $\leb(G_s)> q_s$ we find some hyperarithmetic open set $U_s\subset G_s$ of measure exactly~$q_s$ and let $C^{\alpha}_{s+1} = C^\alpha_s\cup U_s$. It is easy to check that the bound $\leb(C^\alpha)\le \epsilon \+\mu(|\alpha|)$ is maintained at stage~$s+1$; 
 that if $\leb(G_s)> q_s$ then $\epsilon \+\mu_s(|\alpha|) - \leb(C^\alpha_{s+1}) \le \delta_\alpha 2^{-p(s)}$; and that that $\leb(E^\alpha_s)\le \delta_\alpha 2^{-p(s)}$, where $E^\alpha_s = (C^{\alpha}_{s+1}\setminus C^\alpha_s) \cap C_s$.

Suppose that $Z\notin \bigcup_{\beta\prec A}C^\alpha$; since~$\Phi$ is consistent on~$Z$ and $C^\beta\subseteq \Phi^{-1}[\beta]$, $Z\notin C = \bigcup_{\beta\in 2^{<\w}}C^\beta$. Let $\alpha\prec A$, and suppose for a contradiction that $\leb(C^\alpha)< \epsilon \+\mu(|\alpha|)$; let $ \leb(C^\alpha) < r< q< \epsilon \+\mu(|\alpha|)$ be rational numbers. For all but a bounded set of stages~$s$ we have $\epsilon\+\mu_s(|\alpha|)>q$, $\leb(C^\alpha_s) < r$, and $\delta_\alpha 2^{-p(s)} < q-r$. Suppose that~$s$ is a late stage at which~$\alpha$ is considered; so $Z\in \Phi^{-1}[\alpha]$. The fact that $Z\notin C$ implies that $\leb(G_s)>q_s$, but then enough measure is added to $C^\alpha_{s+1}$ to bring it to within $\delta_\alpha 2^{-p(s)}$ of $\epsilon\+\mu_s(|\alpha|)$; this is a contradition, which yields~(2). 

It remains to verify~(3). For each $\alpha$ let $E^\alpha = \bigcup_s E^\alpha_s$; so $\leb(E^\alpha)\le \delta_\alpha$. The sets $C^\alpha\setminus E^\alpha$ are pairwise disjoint: a real $X\in C^\alpha \setminus E^\alpha$ enters~$C^\alpha$ before it enters any other~$C^\beta$. Hence 
\[ \sum_{\alpha\in 2^{<\w}} \leb(C^\alpha) = \sum_{\alpha\in 2^{<\w}} \leb(C^\alpha\setminus E^\alpha) + \sum_{\alpha\in 2^{<\w}} \leb(E^\alpha) \le 1 + \sum_{\alpha\in 2^{<\w}} \delta_\alpha
\] which is finite.


\subsection{The main lemma} 
\label{sub:golden_run_the_main_lemma}

Unlike the hungry sets construction, there are no major topological complications associated with the golden run argument. The proof translated to the higher setting without many modifications. \Cref{prop:K_trivials_are_collapsing} gives a useful approximation with which to run the construction. In the standard construction we assume that the given enumeration is first sped-up so that at every stage~$s$, $A_s\rest{s}$ is $K_s$-trivial; here we can assume that~$A_s$ in its entirety is~$K_s$-trivial. When drip-feeding measure we are instructed to put some weight on a fresh number~$n$, and this usually means larger than any number chosen so far. This of course we cannot do. However we can choose a number as large as necessary (larger than the length of some initial segment of~$A$ which we are trying to certify) without needing to re-use followers; at stage~$s$ we choose from the $p(s)\tth$ column of~$\w$. 

This allows us to prove the higher version of the main lemma \cite[Lemma 5.5.1]{Nies2009}. Suppose that $\seq{A_s}_{s<\wock}$ is a $\wock$-computable approximation of a set~$A$. For $s<\wock$ let $A_{s}\wedge A_{s+1}$ be the longest common initial segment of~$A_s$ and~$A_{s+1}$. Let~$\mu^A$ be a higher $A$-c.e.\ discrete measure. If~$\seq{A_s}$ is a collapsing approximation then we may assume that we have an enumeration $\seq{\mu_s}$ of~$\mu$ such that for all $s<\wock$, $\mu_s^{A_s}$ is a higher c.e.\ discrete measure as well (in fact as discussed above we may assume that~$\mu^X$ is a discrete measure for all oracles~$X$). Recall that for a discrete measure~$\nu$ we let $\nu(\w) = \sum_n \nu(n)$. The quantity 
\[
	\mu_s^{A_s}(\w) - \mu_s^{A_s\wedge A_{s+1}}(\w)
\]
is the total mass assigned by~$\mu_s^{A_s}$ which was believed at stage~$s$ but thought to be incorrect at stage~$s+1$. 

\begin{proposition} \label{prop:main_lemma_of_K_triviality}
	Let~$A$ be higher $K$-trivial, and suppose that~$\mu^A$ is a higher $A$-c.e.\ discrete measure. Then there is an approximation $\seq{A_s}$ of~$A$ such that the sum
	\[
		\sum_{s<\wock}	\left(\mu_s^{A_s}(\w) - \mu_s^{A_s\wedge A_{s+1}}(\w)\right)
	\]
	is finite. 
\end{proposition}
Further, we may assume that if~$\seq{A_s}$ is a given collapsing approximation of~$A$ and $\seq{\mu_s}$ is an enumeration of~$\mu$ such that for all~$s$, $\mu_s^{A_s}(\w)\le 1$, then there is a $\wock$-computable closed unbounded set $C\subseteq \wock$ such that 
	\[
		\sum_{s\in C}	\left(\mu_s^{A_s}(\w) - \mu_s^{A_s\wedge A_{s^+}}(\w)\right) < \infty,
	\]
	where $s^+ = \min (C \setminus (s+1))$.

\smallskip

We obtain familiar corollaries:
\begin{itemize}
	\item Every higher $K$-trivial set is low for higher~$K$; this completes the proof of \cref{thm:higher_K-trivial_coincidence}.
	\item Every higher $K$-trivial set is higher Turing reducible to a higher c.e., higher $K$-trivial set.
	\item Every higher $K$-trivial set is higher $\w$-c.a.
\end{itemize}


\section{Higher weak 2-randomness}

Recall that a higher weak 2-test (a generalised higher ML test) is a sequence $\seq{U_n}$ of uniformly $\Pi^1_1$ open sets (higher c.e.\ open sets) whose intersection is null. Note that we can suppose that the $U_n$ are nested, i.e., $U_{n+1} \subseteq U_n$ for all~$n$ (indeed, if they are not, one can consider $V_n = \bigcap_{k \leq n} U_k$ and observe that the $V_n$ are nested and that their intersection is the same as $\bigcap_n U_n$). 

A sequence is higher weak 2-random if it avoids all higher weak 2-tests. In this section we find alternative, Demuth-like characterisations of higher weak 2-randomness; we consider their Borel rank through an effective lens; and we investigate the interaction with classes of higher $\Delta^0_2$ sequences. These considerations will culminate in a separation of $\Pi^1_1$ randomness from higher weak 2-randomness.

\subsection{Compact approximations and higher weak 2-randomness} 

\Cref{def:compact_approximation} describes compact approximations. We recall the notational convention discussed in \cref{rmk:the_index_wock}: if $\seq{f_s}_{s<\wock}$ is a $\wock$-computable approximation of a function~$f$ then we write $f_{\wock}$ for~$f$. 

\begin{proposition} \label{prop:w2r_closed_approx}
No sequence $X \in 2^\omega$ with a higher closed approximation is higher weakly $2$-random.
\end{proposition}

\begin{proof}
Let $\seq{X_s}_{s \le \wock}$ be a closed approximation of $X = X_{\wock}$. Let  $C = \{X_s\,:\, s\le \wock\}$. We let $U_n = \bigcup_{s<\wock}[X_s \rest n]$. The sequence $\seq{U_n}$ is uniformly higher effectively open. Certainly $X\in \bigcap_n U_n$. If $Y\in U_n$ then the distance of $Y$ from $C$ is at most $2^{-n}$. Hence if $Y\in \bigcap_n U_n$ then the distance of~$Y$ from~$C$ is~$0$. Since~$C$ is closed, this implies that $\bigcap_n U_n \subseteq C$. 

The set~$C$ is countable, and so null. This shows that $\bigcap_n U_n$ is null, and so is a higher weak 2-test. 
\end{proof}

Even if $\seq{X_s}$ is a higher left-c.e.\ approximation, we do not know how to directly show that the measure of the sets~$U_n$ tends to~$0$. 

\medskip

A generalisation of \cref{prop:w2r_closed_approx} gives a Demuth-style characterisation of higher weak 2-randomness, a weakening of the class higher $\MLR[O]$ (introduced later in \cref{section:mlr_plop_o}). In the lower setting of course weak 2-randomness is equivalent to $\MLR[\emptyset']$. Recall that we let $W_e$ denote the $e\tth$ higher c.e.\ open set. 

\begin{proposition} \label{prop:Demuth_characterisations_of_weak_2_randomness}
The following classes of tests precisely capture higher weak 2-tests.
\begin{enumerate}
	\item \label{item:finite-change} Nested tests of the form $\seq{W_{f(n)}}$ where $\leb(W_{f(n)})\le 2^{-n}$ and $f$ has a finite-change approximation. 
		\item \label{item:compact} Nested tests of the form $\seq{W_{f(n)}}$ where $\leb(W_{f(n)})\le 2^{-n}$ and $f$ has a compact approximation. 
\end{enumerate} %
\end{proposition}

\begin{proof}
Every function which has a finite-change approximation also has a compact approximation (\cref{lem:finite-change_implies_compact}). So we need to show that:
\begin{enumerate}
	\item[(a)] Every weak 2-test can be covered by a test with a finite-change index function (as in (\ref{item:finite-change})). 
	\item[(b)] Every test with a compact index function (as in (\ref{item:compact})) can be covered by a weak 2-test.
\end{enumerate}
	
For (a), let $\seq{U_n}$ be a higher weak 2-test; let $U_{n,s}$ be a uniform enumeration of $U_n$. For $s\le \wock$ let $f_s(k)$ be the least $n$ such that $\leb(U_{n,s})\le 2^{-k}$. Since the measures of $U_{n,s}$ are non-decreasing, the functions $f_s(k)$ are non-decreasing in~$s$, and converge to a limit since for all~$k$ there is an~$n$ such that $\lambda(U_n) < 2^{-k}$. So $\seq{f_s}$ is a finite-change approximation of $f=f_{\wock}$. Passing to canonical indices we get a test with a finite-change index function which covers the test $\seq{U_n}$. 

For (b), the argument is inspired by that of \cref{prop:w2r_closed_approx}. Let $\seq{f_s}_{s<\wock}$ be a compact approximation of a function~$f$ such that $\leb(W_{f(n)})\le 2^{-n}$ and $\seq{W_{f(n)}}$ is nested. 

A priori, the sets $W_{f_s(n)}$ (for a fixed~$s$) may not be nested. We replace $W_{f_s(n)}$ by $\bigcap_{m\le n} W_{f_s(m)}$. This changes the index function. However the first~$n$ values of the new index function~$g_s$ are determined by the first $n$ bits of $f_s$. In particular, the map $f_s\mapsto g_s$ is continuous, and hence the set $\{g_s\,:\, s\le \wock\}$ is compact (and of course $g = \lim_s g_s$). Thus, without loss of generality, we may assume that each test $\seq{W_{f_s(n)}}$ is nested. We may also assume that $\leb(W_{f_s(n)})\le 2^{-n}$ for all $s$ and $n$. 

Let $U_n = \bigcup_{s<\wock} W_{f_s(n)}$. Since $U_n \supseteq W_{f(n)}$, the test $\seq{U_n}$ covers the given test $\seq{W_{f(n)}}$; and the sets $U_n$ are uniformly $\Pi^1_1$ open. We show that $\bigcap_n U_n$ is null. 

For each $s\le \wock$, let $A_s = \bigcap_n W_{f_s(n)}$, and let $A = \bigcup_{s\le \wock} A_s$. Each~$A_s$ is null; since $\wock+1$ is countable, $A$ is null. We show that $\bigcap_n U_n\subseteq A$. For let $Y\in \bigcap_n U_n$. For each $n$ there is some $s(n)$ such that $Y\in W_{f_{s(n)}(n)}$. Since the set $\{f_t\,:\, t\le \wock\}$ is compact, the set $\{f_{s(n)}\,:\, n<\w\}$ has a limit point, and that limit point equals~$f_t$ for some $t\le \wock$. Then $Y\in A_t$: to see this, let $n<\w$. There is some $k> n$ such that $f_t\rest {n+1} = f_{s(k)}\rest{n+1}$. Then $Y\in W_{f_{s(k)}(k)}\subseteq W_{f_{s(k)}(n)} = W_{f_t(n)}$ as required. 
\end{proof}

\subsection{A short proof of a theorem of Chong and Yu's} 

Chong and Yu \cite{ChongYu} showed that every hyperdegree above that of Kleene's $O$ contains a higher ML-random set which is not higher weak 2-random. The above results give us a short proof of this fact. Let $Y\ge_h O$. There is some $X\equiv_h Y$ such that $X \ge_{\hT} O$, for example $X = Y\oplus O$. By the higher Ku\v{c}era-G\'acs theorem there is some $Z\equiv_{\hT} X$ which is higher ML-random. Since $Z\ge_{\hT} \Omega$ and higher~$\Omega$ is not higher weak 2-random, neither is $Z$ (\cref{thm:downward_closure_Strong-ML-randomness}). And $Z\equiv_h Y$.

\subsection{The effective Borel rank of higher weak 2-randomness} 

Every higher null weak 2-set is $G_\delta$, and so the set of higher weak 2-random sequences is $\boldface{\Pi^0_3}$. Yu showed that this is sharp. There is a natural higher lightface version of the Borel hierarchy. For example a set is higher $\Pi^0_2$ if it is the uniform intersection of $\Pi^1_1$ open sets (so the higher null weak 2-sets are precisely the null higher $\Pi^0_2$ sets). A set is higher $\Sigma^0_3$ if it is the uniform union of higher $\Pi^0_2$ sets, and so on. We investigate this hierarchy in detail in \cite{Pi11RandomnessPaper}. Here we show that the set of higher weakly 2-random sequences is not higher $\Pi^0_3$. Thus, picking out the null higher $\Pi^0_2$ sets requires an oracle. This follows from \cref{prop:w2r_closed_approx,lem:finite-change_implies_compact} and the following proposition. 

\begin{proposition} \label{prop:finite-change_in_conull_Pi3}
	Every higher $\Pi^0_3$ set of measure~$1$ contains a sequence which has a finite-change approximation.
\end{proposition}

\begin{proof}
Let $F$ be a higher $\Pi^0_3$ set of measure~$1$. So $F = \bigcap_{e<\w} F^e$, where $F^e$ are uniformly higher $\Sigma^0_2$, and since $F\subseteq F^e$, each $F^e$ has measure~$1$. We write $F^e = \bigcup_k F^{e,k}$ where $\seq{F^{e,k}}_{k<\w}$ is an increasing sequence of uniformly higher effectively closed sets, namely, $\Sigma^1_1$ closed sets. 

We define a real $x\in F$ by recursion on $e<\w$. To ensure that $x\in F$ we will, for each~$e$, pick one of the closed sets $F^{e,k}$ and ensure that $x\in F^{e,k}$. We denote the index~$k$ chosen by $c(e)$. We define $x\rest{e}$ and $c\rest{e}$ by simultaneous recursion. At step $e<\w$, given $x\rest{e}$ and $c\rest{e}$, let $H^{e} =\bigcap_{d<e} F^{d,c(d)}$. For $e=0$ we have $H^e = 2^\w$. Inductively we ensure that $\leb(H^e\given x\rest e) \ge 2^{-e}$. We then choose:
\begin{itemize}
	\item  $x(e)\in \{0,1\}$ to be the least so that $\leb(H^e\given x\rest{e+1})\ge 2^{-e}$. 
	\item Since $F^{e}$ has measure~$1$, $\leb(H^e\cap F^{e}\given  x\rest{e+1})\ge 2^{-e}$, and so there is some $k<\w$ such that $\leb\left(H^e\cap F^{e,k} \given  x\rest{e+1} \right)\ge 2^{-(e+1)}$. We let $c(e)$ be the least such~$k$. 
\end{itemize}

For all $e<\w$ and $d\ge e$, $H^e\cap [x\rest d] \supseteq H^d\cap [x\rest{d}]$ are not null and so nonempty. Since $H^e$ is closed, $x\in H^e$. And $H^e\subseteq F^{d}$ for all $d < e$, and so $\bigcap H_e\subseteq F$. Thus $x\in F$. 

It remains to show that~$x$ has a finite-change approximation. To do so, we approximate the set~$F$ and the sets it is built up from. The sets $F^{e,k}$ have (uniform) co-enumerations $F^{e,k}_s$ for $s<\wock$; each $F^{e,k}_s$ is hyperarithmetic and if $s<t$ then $F^{e,k}_s\supseteq F^{e,k}_t$. We also assume that these co-enumerations are continuous: for limit $s<\wock$, $F^{e,k}_s = \bigcap_{t<s} F^{e,k}_t$. We let $F^e_s = \bigcup_k F^{e,k}_s$. We then repeat the construction above at each stage $s<\wock$: we define $x_s\in 2^\w$ and $c_s\in \w^\w$ coding choices of indices so that letting $H^e_s = \bigcap_{d<e} F^d_s$ we have:
\begin{enumerate}
	\item $\leb(H^e_s \given  x_s\rest{e}) \ge 2^{-e}$; 
	\item $x_s(e)$ is least such that $\leb(H^e_s \given  x_s\rest{e+1})\ge 2^{-e}$; and
	\item $c_s(e)$ is the least~$k$ such that $\leb(H^{e}_s\cap F^{e,k}_s\given  x_s\rest{e+1})\ge 2^{-(e+1)}$. 
\end{enumerate}

We will show that $\seq{x_s}$ is a finite-change approximation of~$x$. To begin, we note that if $e<\w$, $s<t\le \wock$ and $c_s\rest{e} = c_t\rest{e}$ then $H^e_s\supseteq H^e_t$. This implies the following:

\begin{itemize}
	\item[(*)] Suppose that $c_s\rest {e} = c_t\rest{e}$ and $x_s\rest{e} = x_t\rest{e}$. Then $x_s(e)\le x_t(e)$. 
	\item[(**)] Suppose that $c_s\rest {e} = c_t\rest{e}$ and $x_s\rest{e+1} = x_t\rest{e+1}$. Then $c_s(e)\le c_t(e)$. 
\end{itemize}

The following claim shows that we cannot cycle through infinitely many values of $c_r(e)$ while $c_r\rest e$ remains stable. We use the following notation. If $I\subseteq \wock$ is an interval of stages and $x_r\rest e$ is constant for all $r\in I$, then we denote this constant value by $x_I\rest e$; similarly for~$c$, or $x(e)$, etc.

\begin{claim} \label{subclaim:finite_change_in_conull_Pi3}
	Let $e<\w$. Let~$I\subseteq \wock$ be an interval of stages on which $c_r\rest{e}$ and $x_r\rest{e}$ are constant. Then $c_{\sup I}\rest{e} = c_I\rest{e}$ and $x_{\sup I}\rest{e} = x_I\rest{e}$.
\end{claim}

\begin{proof}
	By induction on~$e$. Assume we know this for $e$. Let $s = \sup I$. We assume that $c_r\rest{e+1}$ and $x_r\rest{e+1}$ are constant on~$I$; we need to show that $x_s(e) = x_I(e)$ and $c_s(e) = c_I(e)$. By induction and continuity of the co-enumeration of the closed sets $F^{e,k}$, $H^e_s = \bigcap_{r\in I} H^e_r$. For all $r\in I$, $c_r(e)$ is the least $i\in \{0,1\}$ such that $\leb( H^e_r \given x_I\rest{e}\conc i)\ge 2^{-e}$. By induction, $x_s\rest e = x_I\rest{e}$, and by continuity, $\leb(H^e_s \given x_I\rest{e}\conc i) = \inf_{r\in I} \leb( H^e_r\given x_I\rest e\conc i)$ and so is at least $2^{-e}$. On the other hand, if $i=1$, then $\leb(H^e_r \given x_I\rest{e}\conc 0)< 2^{-e}$ for all $r\in I$, and so $\leb(H^e_s \given x_s\rest{e}\conc 0)< 2^{-e}$. Overall we see that $x_s(e) = x_I(e)$. The argument for $c_s(e)$ is the same.
\end{proof}

We show that $\seq{x_s}$ changes only finitely often on each input. \Cref{subclaim:finite_change_in_conull_Pi3} would then imply that $x = \lim_{s\to \wock} x_s$. By induction on~$e$ we show that $\wock+1$ can be partitioned into finitely many closed intervals of stages on which both $x_s\rest e$ and $c_s\rest e$ are constant. Suppose that this has been shown for~$e$; let~$I$ be a closed interval of stages on which $x_s\rest e$ and $c_s\rest e$ are constant. For $i<2$ let $I_i$ be the set of stages $s\in I$ at which $x_s(e) = i$. By (*), both $I_0$ and $I_1$ are intervals, with $I_0<I_1$. \Cref{subclaim:finite_change_in_conull_Pi3} shows that they are closed. Now fix $i<2$; let $t = \max I_i$ and let $k = c_t(e)$. For $m\le k$ let $I_{i,m}$ be the set of stages at which $c_s(e)=m$. By (**), each $I_{i,m}$ is an interval with $I_{i,0}< I_{i,1}< \cdots < I_{i,k}$, and $\bigcup_{m\le k} = I_i$. \Cref{subclaim:finite_change_in_conull_Pi3} shows that each $I_{i,m}$ is closed. This concludes the proof of \cref{prop:finite-change_in_conull_Pi3}.
\end{proof}

\subsection{Separating $\Pi^1_1$ randomness from higher weak 2-randomness} 

In this section we construct a sequence $x\in 2^\w$ which is higher weak 2-random but not $\Pi^1_1$ random. This sequence will be $O$-computable. The construction is an elaboration on that of the previous section. Here too we need to build an element of a $\boldface{\Pi^0_3}$ set of measure~$1$ which is the intersection of higher $\Sigma^0_2$ sets, namely all of the ones of measure~$1$. To ensure that $x$ is not $\Pi^1_1$ random we need to show that it collapses $\wock$, as in the presence of higher weak 2-randomness (and in fact $\Delta^1_1$ randomness), being $\Pi^1_1$ random is equivalent to preserving $\wock$. So we will ensure that we can give~$x$ a collapsing approximation. On the other hand, \cref{prop:w2r_closed_approx} shows that we cannot give~$x$ a compact approximation, let alone a finite-change one. The difficulty of course compared to the previous construction is that we cannot effectively enumerate all of the higher $\Sigma^0_2$ sets of measure~$1$. In the indices of such sets, the property of having measure~$1$ is higher $\Pi^0_1$ but not decidable. 

Technically, it is the key \Cref{subclaim:finite_change_in_conull_Pi3} which may fail: if $F^e$ does not really have measure~$1$, then it is possible that at every stage $r$ in an interval~$I$, $F^e_r$ has measure~$1$, but for $s = \sup I$, $F^e_s$ does not have measure~$1$. (For example, let $I = \w$, and $F^{e,k}_r = 2^\w$ when $r < k$ and empty when $r\ge k$; then $F^{e}_r= 2^\w$ for all $r<\w$ but $F^e_{\w} = \emptyset$). It is then possible that $c_r(e)$ cycles through all of~$\w$ during the stages in~$I$. At stage~$s$ we know that we didn't need to ensure that $x\in F^e$. But by then it is too late, the approximation changed infinitely often. 

Thus, we devise a wider class of approximations which is compatible with being higher weakly 2-random, but still implies collapsing $\wock$. 

\begin{definition} \label{def:finite-change_along_true_path}
	A $\wock$-computable approximation~$\seq{f_s}$ of a function~$f$ is \emph{finite-change along true initial segments} if for no~$n$ is there an increasing infinite sequence~$\seq{t(k)}$ of stages such that $f_{t(k)}\rest n = f\rest n$ for all $k$ but $f_{t(k+1)}(n)\ne f_{t(k)}(n)$ for all~$k$. 
\end{definition}

To see that such an approximation is collapsing we isolate another notion. 

\begin{definition} \label{def:club_approximation}
	An $\wock$-computable approximation~$\seq{f_s}$ of a function~$f$ is a \emph{club approximation} if for all~$n$, the set of stages~$s$ such that $f_s\rest n = f\rest{n}$ is a closed set of stages. 
\end{definition}

\begin{lemma} \label{lem:club_approximations-are_collapsing_etc}
	Every function which has a finite-change-along-true-initial-segments approximation also has a club approximation. If $\seq{f_s}$ is a club approximation of $f\notin \Delta^1_1$ then $\seq{f_s}$ is a collapsing approximation.
\end{lemma}

\begin{proof}
	Suppose that $\seq{f_s}$ is a $\wock$-computable approximation of~$f$ which is finite-change along true initial segments. By induction on~$n<\w$ we see that if~$s$ is a limit stage and for unboundedly many $t<s$, $f_t\rest n = f\rest n$, then $f\rest n = \lim_{t\to s} f_t\rest{n}$. Similarly to what we did with finite-change approximations, we can make the approximation partially continuous by requiring, for every limit stage $s<\wock$ and $n<\w$, that if $\lim_{t\to s} f_t(n)$ exists, then it equals $f_s(n)$. This makes it a club approximation. 

	If $\seq{f_s}$ is a club approximation of~$f$ and~$s$ is least such that~$f$ lies in the closure of $\{f_t\,:\, t<s\}$ then $f = f_s$. Hence if $s<\wock$ then~$f$ is hyperarithmetic. 	
\end{proof}

The separation of $\Pi^1_1$ randomness from higher weak 2-randomness then follows from the following proposition. 

\begin{proposition} \label{prop:separating_Pi11_randomness}
	There is a sequence~$x$ which is higher weak 2-random and has a $\wock$-computable approximation which changes finitely along true initial segments. 
\end{proposition}

The rest of this section is devoted to the proof of \cref{prop:separating_Pi11_randomness}.

\subsubsection{Discussion} 
The new idea is to ``banish'' strings which would contradict the property of the approximation being finite-change along true initial segments. That is, if $I$ is an interval of stages, $x_r\rest e$ is constant on~$I$, but we see $x_r(e)$ changes infinitely often on~$I$, then we require that $x_I\rest e$ is not an initial segment of~$x$. We simply do not allow any future $x_t$ to extend $x_I\rest n$. The construction is dynamic: rather than defining $x$ and $c$ a priori and then giving them approximations, we first define the approximation and then show it converges and has the desired properties.

We need to show that the construction can actually be carried out: at every stage there are non-banished strings that can be chosen to construct $x_s$, particularly non-banished strings relative to which we can make the sets $H^e_s$ not too small. 

This is done as follows:

\medskip

\noindent\textbf{1.} At each length, we will banish at most one string. The continuity properties of the approximations to our sets will ensure that if we do see $c_r(e)$ cycle through all possible values in~$\w$ on an interval~$I$ of stages, then this will witness that $F^{e}$ does not have measure~$1$. Once we see that, we no longer need to force~$x$ to enter $F^{e}$ (we can replace $F^{e}$ by $2^\w$). After this event there will be no need to banish another string of length~$e$. 

\medskip

\noindent\textbf{2.} Nonetheless, even if just one string is banished, it is possible that this was the string on which $H^e$ was large. I.e., it is possible that one string of length $e+1$ is banished and the other is useless. To counter this we rely on a measure-theoretic observation which is the basis of Ku\v{c}era's coding technique \cite{Kuvcera1985}. We spread out the levels of the construction, adding more than one bit between step~$e$ and $e+1$. If the levels are sufficiently spread out, then every good string at level~$e$ has at least two good strings at the next level. So if one of them is banished, the other can still be used. 

\medskip

These are the ideas needed for the construction. We can now give the formal details. 

\subsubsection{Construction}

We start with an effective enumeration $\seq{F^e}$ of all higher $\Sigma^0_2$ sets. So $F^e = \bigcup_k F^{e,k}$, an increasing sequence, with each $F^{e,k}$ a closed $\Sigma^1_1$ set. Each of these have co-enumerations $\seq{F^{e,k}_s}_{s<\wock}$. We let $F^e_s = \bigcup_{k<\w} F^{e,k}_s$. If~$s$ is a limit ordinal then $F^{e,k}_s = \bigcap_{t<s} F^{e,k}_t$. 

We require that $F^{0,k} =2^\w$ for all~$k$. 

Let $\seq{\ell(e)}_{e<\w}$ and $\seq{\epsilon^e}_{e<\w}$ be computable sequences such that:
\begin{itemize}
	\item $\seq{\ell(e)}$ is an increasing sequence of natural numbers with $\ell(0) = (0)$. 
	\item $\seq{\epsilon^e}$ is a decreasing sequence of positive rational numbers with $\epsilon^0 = 1$. 
	\item For any $e<\w$, for any measurable set~$A$, and for any string $\s$ of length $\ell(e)$, if $\leb(A\given\s)\ge \epsilon^e/2$ then there are at least two extensions $\tau$ of $\s$ of length $\ell({e+1})$ such that $\leb(A\given\tau)\ge \epsilon^{e+1}$. 
\end{itemize}

If $\leb(F^e)<1$ then we may define during the construction a string $\rho^e$ of length $\ell(e)$; this will be the ``banished'' string of length~$\ell(e)$. We will ensure that the real we build does not extend $\rho^e$. [We required $F^0$ to have measure~$1$ to ensure that $\rho^0$ is never defined, as we would not have been able to avoid it.]

\

At every stage~$s$ we will define:
\begin{itemize}
	\item A sequence $x_s\in 2^\w$; 
	\item A sequence of closed sets $\seq{H^e_s}$;
	\item A function $c_s\in (\w+1)^\w$ which codes our choices which define the closed sets $H^e_s$. [A choice $k = c_s(e)<\w$ indicates as before the choice of $F_s^{e,k}$; $c_s(e)=\w$ indicates that $\leb(F^e_s)<1$.]
\end{itemize}

At a limit stage~$s<\wock$ we first see if we need to banish some strings. Let $e<\w$ and suppose that  $\leb(F^e_s)<1$ but that there is some final segment $I = [s_0,s)$ of~$s$ such that 
\begin{itemize}
	\item $\leb(F^e_r)=1$ for all $r\in I$;
	\item $c_r\rest{e}$ is constant on the interval $I$; and
	\item The string $x_r\rest{\ell(e)}$ is constant on $I$.
\end{itemize}
Then we define $\rho^e = x_I\rest{\ell(e)}$. We do this for all~$e$ for which this is needed. Note that $\leb(F^e_t)$ is nonincreasing in~$t$, and so for all~$e$ there may be at most one stage at which we want to define~$\rho^e$.

\medskip

We then define $x_s$, our choice function~$c_s$ and the closed sets $H^e_s$. To start, we let $H^0_s = 2^\w$. At step~$e$ we already have $x_s\rest{\ell(e)}$, $c_s\rest{e}$ and $H^e_s$. By induction, $\leb(H^e\given x\rest{\ell(e)})\ge \epsilon^e \,[s]$. 

At step~$e$ of stage~$s$ we first define $c_s(e)$ and $H^{e+1}_s$:
\begin{itemize}
	\item If $\leb(F^e_s) = 1$ then we let $c_s(e)$ be the least $k<\w$ such that $\leb(H^e\cap F^{e,k}\given  x\rest{\ell(e)})\ge \epsilon^{e}/2\,[s]$. We then let $H^{e+1}_s = H^e_s\cap F^{e,c_s(e)}_s$.
	\item If $\leb(F^e_s)<1$ then we let $c_s(e)= \w$ and $H^{e+1}_s = H^e_s$. 
\end{itemize}

We then define $x_s\rest{\ell(e+1)}$:
\begin{itemize}
	\item If $\rho^{e+1}$ is undefined then we let $x_s\rest{\ell(e+1)}$ be the leftmost extension $\s$ of $x_s\rest{\ell(e)}$ of length $\ell({e+1})$ such that $\leb(H^{e+1}_s\given\s)\ge  \epsilon^{e+1}$. 
	\item If $\rho^{e+1}$ is defined then we let $x_s\rest{\ell(e+1)}$ be the leftmost extension $\s$ of $x_s\rest{\ell(e)}$ of length $\ell({e+1})$ \emph{other than $\rho^{e+1}$} such that $\leb(H^{e+1}_s\given\s)\ge  \epsilon^{e+1}$. 
\end{itemize}

This concludes the construction.

\

\subsubsection{Verification}

As above, if $e<\w$, $s<t\le \wock$ and $c_s\rest{e} = c_t\rest{e}$ then $H^e_t\subseteq H^e_s$. This implies:
\begin{itemize}
	\item[(*)] Suppose that $c_s\rest {e} = c_t\rest{e}$ and $x_s\rest{\ell(e)} = x_t\rest{\ell(e)}$. Then $c_s(e)\le c_t(e)$. 
	\item[(**)] Suppose that $c_s\rest {e+1} = c_t\rest{e+1}$ and $x_s\rest{\ell(e)} = x_t\rest{\ell(e)}$. Then $x_s\rest{\ell(e+1)}\le x_t\rest{\ell(e+1)}$ (lexicographically).
\end{itemize}
For (**) note that if $\rho^{e+1}$ is first defined between stages~$s$ and~$t$, this only pushes $x_t\rest{\ell(e+1)}$ further to the right. For (*) again note that if $\leb(F^e_t)=1$ then $\leb(F^e_s)=1$.

The following claim shows that banishing conforms to out original intention. Suppose that $c_r\rest{e}$ and $x_r\rest{\ell(e)}$ are constant on an interval~$I$ of stages. Suppose that $c_r(e)<\w$ for all $r\in I$. By (*), $\sup_{r\in I}c_r(e) = \w$ if and only if $c_r(e)$ changes infinitely often on~$I$ (there is an infinite increasing sequence $\seq{t(k)}$ of stages in~$I$ such that $c_{t(k+1)}(e)\ne c_{t(k)}(e)$). 

\begin{claim} \label{clm:banishing_works}
	Let $s<\wock$ be a limit stage. Let $e<\w$. Suppose that both $c_r\rest{e}$ and $x_r\rest{\ell(e)}$ are constant on a final segment~$I$ of~$s$. Suppose that $c_r(e)<\w$ for all $r\in I$ but that $\sup_{r\in I} c_r(e)=\w$. Then at stage~$s$ we define $\rho^e = x_I\rest{\ell(e)}$. 
\end{claim}

\begin{proof}
	Let $\s = x_I\rest{\ell(e)}$.
If $r<t$ are in $I$ then $H^e_t\subseteq H^e_r$. Let $H^e_{<s} = \bigcap_{t\in I} H^e_t$. 
	
	If $t\in I$ and $k< c_t(e)$ then $\leb(H^e_t\cap F^{e,k}_t\given \s) < \epsilon^{e}/2$, and so $\leb(H^e_{<s}\cap F^{e,k}_s\given\s)<\epsilon^{e}/2$. It follows that $\leb(H^e_{<s}\cap F^{e}_s\given\s)\le \epsilon^e/2$. 

	On the other hand, for all $t\in I$, $\leb(H^e_t\given\s)\ge \epsilon^e$ and so $\leb(H^e_{<s}\given\s)\ge \epsilon^e$. This shows that $\leb(F^e_s)<1$. The conditions for defining $\rho_e = \s$ at stage~$s$ are fulfilled.
\end{proof}

Since each string of length $\ell(e)$ has only finitely many extensions of length $\ell(e+1)$, (**) and \cref{clm:banishing_works} together imply:

\begin{claim} \label{clm:banishing_works2}
	Let $s<\wock$ be a limit stage. Let $e<\w$. Suppose that both $c_r\rest{e}$ and $x_r\rest{\ell(e)}$ are constant on a final segment~$I$ of~$s$. Suppose that $x_r\rest{\ell(e+1)}$ changes infinitely often on~$I$ (but not on a proper initial segment of~$I$). Then at stage~$s$ we define $\rho^e = x_I\rest{\ell(e)}$. 
\end{claim}

By induction on~$e$ we can show that eventually each $x_s\rest{\ell(e)}$ and $c_s\rest{e}$ are constant. We can let $x = \lim_{s\to \wock} x_s$ and $c = \lim_{s\to \wock} c_s$. 

\begin{claim} \label{clm:weak2random}
	$x$ is higher weak 2-random.
\end{claim}

\begin{proof}
	Let $e<\w$, and suppose that $\leb(F^e)=1$. We show that $x\in F^e$. Let~$I$ be a final segment of~$\wock$ on which $c_t\rest{e+1}$ is constant. Since $\leb(F^e)=1$, $k = c(e) <\w$. For all $t\in I$ and all $d\ge e$, $H^{d}_t\subseteq F^{e,k}_t$; and $[x_t\rest{\ell(d)}]\cap H^d_t$ is not null, and so nonempty. It follows that $[x\rest{\ell(d)}]\cap F^{e,k}$ is nonempty. We then use the fact that~$F^{e,k}$ is closed. 
\end{proof}

The proof of \ref{prop:separating_Pi11_randomness} is concluded by showing that $\seq{x_s}$ is an approximation which changes finitely often along true initial segments. To see this, it suffices to show that for no $e<\w$ is there an increasing sequence $\seq{t(k)}$ of stages such that $x_{t(k)}\rest{\ell(e)} = x\rest{\ell(e)}$ for all~$k<\w$, but that $x_{t(k+1)}\rest{\ell(e+1)}\ne x_{t(k)}\rest{\ell(e+1)}$ for all $k<\w$ (to verify \cref{def:finite-change_along_true_path} for an arbitrary~$n$, consider the greatest~$e$ such that $\ell(e)\le n$). Suppose that such a sequence $\seq{t(k)}$ is given; let $s = \sup_k t(k)$. Let~$d$ be the greatest such that both $\lim_{r\to s} x_r\rest{\ell(d)}$ and $\lim_{r\to s} c_r\rest{\ell(d)}$ exist. So $d\le e$. Either the conditions of \cref{clm:banishing_works} or \cref{clm:banishing_works2} hold at stage~$s$ for $d$, so at stage~$s$ we define $\rho^d = x\rest{\ell(d)}$. However the construction ensures that for all $d\ge 1$, $x$ does not extend $\rho^d$ (and that $\rho^0$ is never defined).

\section{Classes of higher $\Delta^0_2$ functions}

Motivated by the their usage in investigating higher weak 2-randomness, we study the classes of higher $\Delta^0_2$ functions which we introduced above. We first consider the higher limit lemma.

\subsection{The higher limit lemma} \label{subsec:higher_limit_lemma}

The proof of the higher limit lemma (\cref{prop:higher_limit_lemma}) is not complicated. The equivalence of $f\le_{\Tur}O$ and $f\le_{\hT}O$ was established in \cref{prop:O_collapses_higher_computability}. If $f = \Phi(O)$ (where $\Phi$ is either c.e.\ or higher c.e.) then we can give $f$ a $\wock$-computable approximation by letting $f_s = \Phi_s(O_s)$, where $\seq{O_s}$ is a $\wock$-computable enumeration of~$O$. And if $\seq{f_s}$ is a $\wock$-computable approximation of~$f$ then the graph of~$f$ is $\Delta_2$ over $L_{\wock}$; since a set is $\Sigma_2$ over $L_{\wock}$ if and only if it is c.e.\ in~$O$, we see that~$f$ is $O$-computable. 

In fact, the higher limit lemma relativises to every oracle. Recall that a subset~$X$ of~$L_{\wock}$ is $A$-$\wock$-computable (where $A\in 2^\w$) if there is a $\wock$-c.e.\ $\Phi\subseteq 2^{<\w}\times L_{\wock}$ such that $X = \Phi(A)$. Also recall that we let $J^A$ be the higher jump of~$A$, the effective join of all subsets of $\w$ which are higher $A$-c.e.
 
\begin{proposition}
Let $A\in 2^\w$. The following are equivalent for $f\colon \w\to \w$. 
\begin{enumerate}
\item $f \le_{\hT} J^{A}$. 		
\item $f$ has an $A$-$\wock$-computable approximation $\seq{f_s}_{s<\wock}$.
\end{enumerate}
\end{proposition}

\begin{proof}
Recall that we use a $\wock$-computable projection function $p\colon \wock \to \w$.

\medskip

Assume~(2); Let $m\colon \w\to \wock$ be the modulus of the sequence $\seq{f_s}_{s < \wock}$: The value $m(n)$ is the least $s$ such that for all $t> s$ we have $f_t(n) = f_s(n)$. Let $W = \{ (n,p(s))\,:\, s< m(n)\}$; the set $W$ is higher $A$-c.e.: to enumerate $(n,p(s))$ into $W$, what we need from $A$ is the value $f_s(n)$ and a different value $f_t(n)$ for some $t>s$; both are given with finitely much use of $A$. So $W \le_{\hT} J^A$. Now, from one pair $(n,p(s))\notin W$ and finitely much of $A$ we output $f(n) = f_s(n)$. So $f\le_{\hT} A\oplus W \le_{\hT} J^A$. 

\medskip

Assume~(1). Recall that we regard~$J$ as a higher enumeration functional. The sequence $\seq{J_s^{A\rest{p(s)}}}$ is an $A$-$\wock$-computable approximation of $J^A$ (using the fact that for all~$n$ there is some $t<\wock$ such that $p(s)\ge n$ for all $s\ge t$). Note that the sequence $\seq{J_s^A}$ is \emph{not} $A$-$\wock$-computable. 

If $\Psi$ is a higher Turing functional then $\seq{\Psi_s\left(J_s^{A\rest{p(s)}}\right)}$ is an $A$-$\wock$-computable approximation of $\Psi(A)$. 

We remark that it is not the case that for all~$A$, $J^A$ has an $A$-$\wock$-computable \emph{enumeration} (a $\wock$-computable sequence $\seq{A_s}_{s<\wock}$ such that $A_s \subseteq A_t$ for $s \leq t$).
\end{proof}

\subsection{Equivalent characterisations of classes} 

A couple of classes we defined have equivalent characterisations, some related to the limit lemma.

\subsubsection{Higher $\omega$-computably approximable functions}

These were defined in \cref{def:omega_c.a.}: functions approximable by finite-change approximations which moreover have hyperarithmetic bounds on the number of changes. In complete analogy with the lower case, this notion can be characterised by using strong reducibilities. 
\begin{itemize}
	\item Let $X,Y\in 2^\w$. We say that $X$ is higher truth-table reducible to~$Y$ if there is a hyperarithmetic sequence $\seq{F_n}$ of finite subsets of $2^{<\w}$ such that $X(n)=1$ if and only if $Y$ extends some string in~$F_n$. Nerode's argument shows that $X$ is higher truth-table reducible to~$Y$ if and only if $X=\Phi(Y)$ for some higher turing functional~$Y$ which is total and consistent on all oracles. 
	\item Let $f,g\in \w^\w$. We say that $f$ is higher weak truth-table reducible to~$g$ if there is a higher Turing functional $\Phi$ such that $\Phi(g)=f$ and there is a hyperarithmetic function~$h$ such that for all axioms $(\tau,\s)\in \Phi$, $|\tau|\le h(|\s|)$. 
\end{itemize}

The lower-case arguments carry over to show that $X\in 2^\w$ is higher $\w$-c.a.\ if and only if it is higher truth-table reducible to~$O$; and that $f\in \w^\w$ is higher $\w$-c.a.\ if and only if it is higher weak truth-table reducible to~$O$.

\subsubsection{Finite-change approximations}

As discussed above, a finite-change approximation can be made continuous at limit stages. Hence, $f\in \w^\w$ has a finite-change approximation if and only if it has an approximation $\seq{f_s}$ such that for all limit $s<\wock$, $f_s = \lim_{t\to s} f_t$. 

We give a characterisation using a strong variant of the limit lemma.

\begin{proposition}
The following are equivalent for $f\in \w^\w$:
\begin{enumerate}
\item $f$ has a finite-change approximation.
\item $f$ is higher $O$-computable by a higher Turing functional $\Phi$ which is total (and consistent) on every subset of~$O$.
\end{enumerate}
\end{proposition}

\begin{proof}
(2)$\Then$(1): Let $\seq{O_s}$ be a $\wock$-computable enumeration of~$O$. For $s<\wock$ let $f_s = \Phi(O_s)$. For a limit $s<\wock$, $O_s = \bigcup_{t<s} O_t$. Since~$\Phi$ uses only finitely much of an oracle, $f_s = \lim_{t\to s} f_t$, so~$\seq{f_s}$ is a finite-change approximation of~$f$. 

\

(1)$\Then$(2): this is a modification of the argument that every function which is higher $\w$-c.a.\ is higher weak truth-table reducible to~$O$. Let $\seq{f_s}$ be a finite-change approximation of~$f$. For all~$n$ and $k$ we can compute some $d = d(n,k)$ such that $d\in O$ if and only if there are at least~$k$ changes in $\seq{f_s(n)}$. We then let $\Phi(X,n)=m$ if $m$ is the $k\tth$ value of $f_s(n)$ observed, where $k$ is the least such that $d(n,k)\notin X$. In other words, the procedure~$\Phi$ queries an oracle~$X$ as if it were~$O$, asking successively whether $\seq{f_s(n)}$ changes once, twice, thrice,... until it finds~$X$'s opinion on the number of changes; and outputs the corresponding value. If $X=O$ the answer is correct. If $X\subseteq O$ then the answer could be smaller than the actual number of changes but not larger, so the search for the $k\tth$ value will terminate.
\end{proof}

\subsubsection{Compact approximations}

\begin{lemma} \label{lem:compact_and_countable_closure}
	The following are equivalent for $x\in 2^\w$:
	\begin{enumerate}
		\item $x$ has a closed approximation.
		\item $x$ has a $\wock$-computable approximation $\seq{x_s}$ such that the closure of the set $\{x_s\,:\, s\le \wock\}$ is countable. 
	\end{enumerate}
\end{lemma}

\begin{proof}
	The idea is similar to that of the proof of \cref{lem:compact_implies_collapsing}. In the nontrivial direction, we first note that if $y$ is a limit point of $\{x_s\,:\, s\le \wock\}$ other than~$x$ then there is an increasing sequence $\seq{t(k)}$ of stages such that $y = \lim_{k\to \w} x_{t(k)}$. Further, for all limit $s<\wock$, since the closure of $\{x_t\,:\, t<s\}$ is countable, this closure can be effectively obtained (again using the Cantor-Bendixon analysis). We now fatten the approximation $\seq{x_s}$ by inserting, for each limit $s<\wock$, between $\seq{x_t}_{t<s}$ and $x_s$, all the limit points of $\{x_t\,:\, t<s\}$ which were not previously inserted. If $x_t\rest n$ has stabilised before~$s$, then all limit points extend this string, and so the fattened approximation still approximates~$x$. 
\end{proof}

\subsubsection{Club approximations}

The class of approximations given by \cref{def:locally_almost_finite_change} is mostly a tool which we use later, because it is easier to deal with than club approximations. To motivate that definition we first consider a ``pointwise version''. 

\begin{definition} \label{def:almost_finite_change}
	An approximation $\seq{f_s}$ is \emph{almost finite-change} if for all~$n<\w$, if $\seq{t(i)}$ is an increasing sequence of stages such that $f_{t(i+1)}(n)\ne f_{t(i)}(n)$ for all $i<\w$, then $f_t(n)$ is constant on $[\sup_i t(i),\wock)$. 
\end{definition}

Suppose that an approximation $\seq{x_s}$ consists of elements of Cantor space and that it is partially continuous: for all $n<\w$ and limit $s<\wock$, if $\lim_{t\to s}f_t(n)$ exists then it equals $f_s(n)$. Then the approximation is almost finite-change if and only if for all $n<\w$, for all $s<\wock$ and $i<2$, if the set $\{t<s\,:\, f_t(n) = i\}$ is not a closed subset of~$s$, then $f_s(n)\ne i$. 
 
\begin{definition} \label{def:locally_almost_finite_change}
	An approximation $\seq{f_s}$ is \emph{locally almost finite-change} if for all~$n<\w$ and all strings $\s\in \w^n$, if $\seq{t(i)}$ is an increasing sequence of stages such that $f_{t(i)}\rest n = \s$ and $f_{t(i+1)}(n)\ne f_{t(i)}(n)$ for all $i<\w$, then $f_t(n)$ is constant on the stages $t\ge \sup t(i)$ at which $\s\prec f_t$.
\end{definition}

Call an approximation $\seq{f_s}$ \emph{locally continuous} if for all $n<\w$ and all $\s\in \w^{n}$, the function $f_t(n)$ is continuous on the set of stages~$t$ at which $\s\prec f_t$ (using the subspace topology). Namely, letting $F_\s$ be that set of stages, if $s$ is a limit point of $F_\s$ which is also in $F_\s$, and $f_t(n)$ is constant on a final segment of $s\cap F_\s$, then $f_s(n)$ equals that constant value. 

\begin{lemma} \label{lem:almost_finite_change_and_closed_sets_of_stages}
	Let $\seq{x_s}$ be a locally continuous approximation consisting of elements of Cantor space. Then the approximation is locally almost finite-change if and only if for all strings $\s\in \w^{<\w}$ and all $s<\wock$, if the set $\{t<s\,:\, \s\prec f_t\}$ is not a closed subset of~$s$, then $\s\nprec f_s$.  \qedhere
\end{lemma}

\begin{lemma} \label{lem:club_and_almost_finite_change}
	Every locally almost finite-change approximation is a club approximation. If $x\in 2^\w$ has a club approximation then it has a locally almost finite-change approximation. 
\end{lemma}

\begin{proof}
	Let $\seq{x_s}$ be a club approximation of~$x\in 2^\w$. We may assume it is locally continuous (making it so does not change it being a club approximation). We define a locally continuous sequence $\seq{y_s}$ by recursion. At stage~$s$ we have already defined $\seq{y_t}_{t<s}$. For any string~$\s$ let $F_\s$ be the set of stages~$t$ at which $\s\prec y_t$. 
	
	We call a string $\s$ \emph{forbidden} at stage~$s$ if the set $F_\s\cap s$ is not a closed subset of~$s$. Otherwise a string is \emph{permitted} at stage~$s$. By induction, for all $t<s$, every initial segment of $y_t$ is permitted at stage~$t$. 
		
	The empty string is always permitted. Every string which is permitted at stage~$s$ has an immediate extension which is also permitted. To see this, suppose that~$\s$ is permitted but suppose, for a contradiction, that both $\s\conc 0$ and $\s\conc 1$ are forbidden at stage~$s$. For $i<2$ let $r_i$ be the least stage $r<s$ which is a limit point of $F_{\s\conc i}$ but is not in $F_{\s\conc i}$. Since $y_{r_i}(|\s|)$ has just two possible values, $r_0\ne r_1$. Say $r_0<r_1$. But this means that $\s\conc 0$ is forbidden at stage~$r_1$, so by induction we cannot have $\s\conc 0\prec y_{r_1}$, a contradiction.

	We define $y_s$ by induction. Suppose that $\s = y_s\rest n$ is defined; by induction this string is permitted at stage~$s$. We then act as follows:
	\begin{enumerate}
		\item If one extension $\s\conc i$ is forbidden at stage~$s$ then we let $y_s(n) = 1-i$.
		\item Otherwise, we let $y_s(n) = x_s(n)$. 
	\end{enumerate}
	
	The fact that $\seq{x_t}$ is locally continuous at~$s$ implies that so is $\seq{y_t}$. Hence, by the construction and by \cref{lem:almost_finite_change_and_closed_sets_of_stages}, the sequence $\seq{y_s}$ is locally almost finite-change. 
	
	By induction on~$s<\wock$ we observe that: (a) no initial segment of~$x$ is forbidden at~$s$; and (b) if $\s$ is an initial segment of both $x$ and $x_s$, then $\s\prec y_s$. We conclude that $x = \lim y_s$. 
\end{proof}

\subsection{Enumerating approximations}  \label{subsec:enumerating_approximations}

In the next subsection we will prove non-implications between classes we defined above. When trying to diagonalise against a class of higher $\Delta^0_2$ functions we need to enumerate an effective list of approximations. We discuss here when this is possible. 

A \emph{partial approximation} is a sequence $\seq{f_t}_{t<s}$ for some $s\le \wock$. 

\begin{lemma} \label{lem:enumerating_partial_approximations}
	There is an effective $\w$-enumeration of all $\wock$-computable partial approximations. That is, there is a partial array $\seq{f^n_t}$ for $n<\w$ and $t<\wock$ such that the function $(n,t)\mapsto f^n_t$ is partial $\wock$-computable, and every $\wock$-computable partial approximation equals $\seq{f^n_t}_{t<s}$ for some $n<\w$. 
\end{lemma}

\begin{proof}
	There is a universal partial $\wock$-computable function. This allows us to devise an array $\seq{f^\alpha_t}$ for $\alpha,t<\wock$ such that every $\wock$-computable partial approximation is $\seq{f^\alpha_t}$ for some $\alpha<\wock$. Now renumber using the projection function~$p$. 
\end{proof}

Uniformly we can totalise approximations: transform a given $\wock$-computable partial approximation $\seq{g_s}$ into a $\wock$-computable approximation $\seq{f_s}_{s<\wock}$ such that if $\seq{g_s}$ is total and converges to some~$g$, then $\lim_s f_s = g$ as well. This is similar to how it is done in lower computability, with care taken at limit stages. Namely, we define a non-decreasing function~$t(s)$ which indicates the next expected~$g_t$. At a successor stage~$s$, if $g_{t(s-1)}$ is revealed by stage~$s$, we let $f_s = g_{t(s-1)}$ and let $t(s) = t(s-1)+1$; otherwise we let $t(s)=t(s-1)$ and $f_s = f_{s-1}$. At a limit stage~$s$ we let $t(s) = \sup_{r<s} t(r)$ and let $f_s(n) = \lim_{t\to s} f_t(n)$ when the limit exists, and 0 otherwise. 

Thus, we can give an $\w$-list of total sequences $\seq{f_s}$, not all of which converge but for which the convergent ones list all higher $\Delta^0_2$ functions. In some cases we can do better. For example, as in the lower case, we can enumerate all higher $\w$-c.a.\ functions:

\begin{lemma} \label{lem:enumerating_higher_omega_ca}
	There is a (total) $\wock$-computable array $\seq{f^n_t}_{n<\w,t<\wock}$ such that:
	\begin{itemize}
		\item For every~$n$, $\seq{f^n_t}_{t<\wock}$ is a higher $\w$-computable approximation of a function~$f^n$. 
		\item Every higher $\w$-c.a.\ function equals $f^n$ for some~$n$. 
	\end{itemize}
\end{lemma}

The construction is as expected. There is a $\wock$-list of all hyperarithmetic functions. Using the projection function~$p$ we can give a partial $\wock$-computable function $n\mapsto h^n$ which enumerates all hyperarithmetic functions. In fact by coupling it with partial approximations we can get a partial $\wock$-computable array $\seq{h^n, g^n_t}$ which lists all pairs $(h,\seq{g_t})$ of hyperarithmetic functions and partial approximations. 

We totalise as above, so we assume that each $\seq{g^n_t}$ is total. We then produce a total approximation~$\seq{f^n_t}$. If $h^n$ is not yet defined at stage~$t$ then $f^n_t$ is the zero function. If $h^n$ is defined at stage~$t$ and for no~$k$ have we seen more than $h^n(k)$ many changes on $\seq{g^n_r(k)}_{r\le t}$ then we let $f^n_t = g^n_t$. Otherwise $f^n_t$ is again the zero function.

\subsubsection{Enumerating other approximations}

On the other hand, many of the classes we listed above cannot be enumerated with corresponding approximations. For example, if $\seq{f^n_t}$ is a list of finite-change approximations, then it does not contain all functions with a finite-change approximation, as direct diagonalisation would verify. Informally, when we try to enumerate only finite change approximations, we track a sequence $\seq{g_t}$ up to a limit stage~$s$ at which we first see infinitely many changes on some input. At each stage $t<s$ we have only seen finitely many changes so we copy~$g_t$. By stage~$s$ we have seen infinitely many changes but it is too late to go back and change the sequence. 

A different difficulty is met when we try to enumerate almost finite-change or locally almost finite-change approximations (\cref{def:almost_finite_change,def:locally_almost_finite_change}). Again diagonalisation shows we cannot list such approximations yielding all functions with these approximations. When we totalise approximations as above, starting with an almost-finite-change approximation we might inadvertently ruin this property. Take such an approximation $\seq{g_t}$ and suppose that the totalising process yields $\seq{f_t}$. Let~$s$ be a limit stage and suppose that at stage~$s$ we have seen $f_t(k)$ change infinitely often. We need to define~$f_s$ but since we are working uniformly, we cannot rely on the fact that $\seq{g_t}$ is total; we cannot wait to see what $g_s(k)$ is; the procedure above has us declare an arbitrary value for $f_s(k)$. When we later see that $g_s(k)$ is different it is too late. Either we change a later value of $f_t(k)$; this means that $\seq{f_t}$ is no longer an almost finite-change approximation. Or we can stick with the value $f_s(k)$; in this case $\seq{f_t}$ is an almost finite-change approximation, but $\lim_t f_t \ne \lim_t g_t$. 

Luckily, for our purposes, we do not need tight restrictions on the kind of approximations we list. We will use the following two listings. 

\begin{lemma} \label{lem:listing_fcatis}
	There is a total $\wock$-computable array $\seq{x^n_t}$ of elements of Cantor space such that:
	\begin{itemize}
		\item For all~$n$, $\seq{x^n_t}$ converges to a real~$x^n$; and
		\item If a real $x\in 2^\w$ has an approximation which changes finitely along true initial segments, then there is some~$n$ such that $x = x_n$ and $\seq{x_n}$ changes finitely along true initial segments.
	\end{itemize}
\end{lemma}

\begin{proof}
	Given a partial approximation, we totalise it to a sequence $\seq{y_s}_{s\le \wock}$ as above, except that at limit stages we make the approximation locally continuous (for limit~$s$ we inductively define $y_s(n)$ to be the limit of $y_t(n)$ over the stages $t<s$ at which $y_t\rest n = y_s\rest n$, 0 if the limit does not exist). If the original approximation changes finitely along true initial segments, so does $\seq{y_s}$. We can then transform the approximation  to be locally almost finite-change, in particular ensuring it has a limit. This follows the construction in the proof of \cref{lem:club_and_almost_finite_change}, tracking forbidden strings. Again, if $\seq{y_s}$ changes finitely along true initial segments, so does the new approximation.
\end{proof}

\begin{lemma} \label{lem:listing_clubs}
	There is a total $\wock$-computable array $\seq{x^n_t}$ of elements of Cantor space such that:
	\begin{itemize}
		\item For all~$n$, $\seq{x^n_t}$ converges to a real~$x^n$. 
		\item Every real $x\in 2^\w$ which has a club approximation equals~$x^n$ for some~$n$. 
	\end{itemize}
\end{lemma}

\begin{proof}
	The idea is to transform partial approximations $\seq{x_t}$ into ``nearly'' locally almost finite-change total approximations. Totalise as above, making the approximation $\seq{y_s}$ locally continuous. Once we have seen, for some~$\s\in 2^{n}$, infinitely many changes in $y_t(n)$ on the set of stages at which $\s\prec y_t$, we set $y_s(n)=0$, but later allow one last change, if we see the value 1 show up in the approximation $\seq{x_t}$. 
\end{proof}

\subsection{Separations between classes} 

None of the classes we defined in the previous sections coincide. For a summary see \cref{fig:Delta_2_classes}. All implications were discussed above. In this section we show that no other implications hold. In fact, all separations are made in Cantor space.

\begin{figure}[h]
	\centering
		\begin{tikzpicture}[every text node part/.style={align=center}]

		\node[draw, ellipse] (WCA) at (0, 8) {$\omega$-c.a.};
		\node[draw, ellipse] (FC) at (4, 8) {finite-change};
		\node[draw, ellipse] (C) at (8, 8) {compact};

		\node[draw, ellipse] (FCTIS) at (4, 6) {finite-change along \\ true initial segments};
		\node[draw, ellipse] (CLUB) at (4, 4) {club};

		\node[draw, ellipse] (COLL) at (4, 2) {collapsing};

		\node[draw, ellipse] (D) at (4, 0) {$\Delta^0_2$ and $\w_1^f>\wock$};

		\draw[thick, -> ] (WCA) -- (FC);
		\draw[thick, -> ] (FC) -- (C);
		\draw[thick, -> ] (FC) -- (FCTIS);
		\draw[thick, -> ] (FCTIS) -- (CLUB);
		\draw[thick, -> ] (CLUB) -- (COLL);
		\draw[thick, -> ] (COLL) -- (D);

		\draw[thick, -> ] (C) to[bend left=60] (COLL);

		\end{tikzpicture}
	\caption{Classes of higher $\Delta^0_2$ reals}
	\label{fig:Delta_2_classes}
\end{figure}
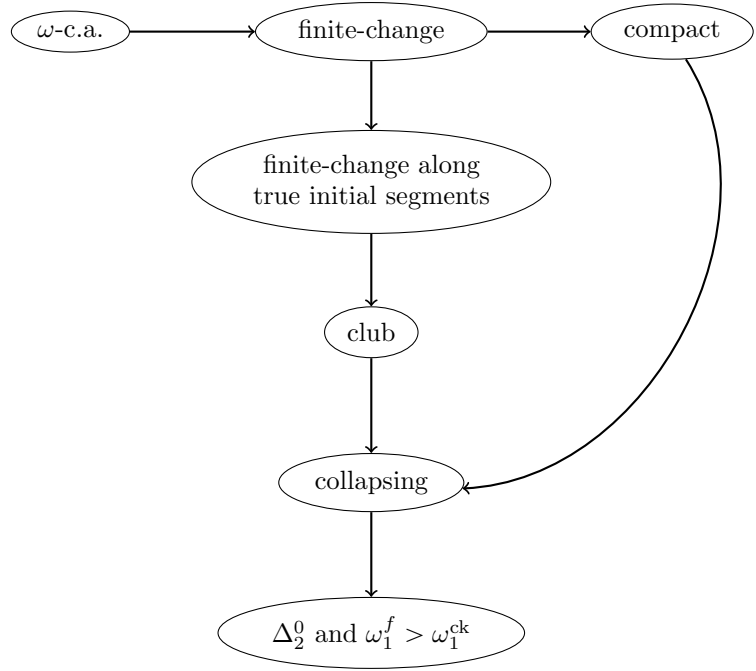

%

\subsubsection{A real with a finite-change approximation which is not $\w$-c.a.}

This is a simple diagonalisation argument, using \cref{lem:enumerating_higher_omega_ca}, but working in Cantor space. Let $\seq{x^n_t}$ be as given by the lemma (with $x^n_t\in 2^\w$). Define $y\in 2^\w$ by letting $y(n) = 1-x^n(n)$. Then $\seq{x^n_t(n)}$ is a finite-change approximation of~$y$.

\subsubsection{A real with a finite-change approximation along true initial segments, but no compact approximation}

An example for such a real is given by \cref{prop:separating_Pi11_randomness} (using \cref{prop:w2r_closed_approx}).

\subsubsection{A higher $\Delta^0_2$ real which collapses $\wock$ but has no collapsing approximation}

In \cite{BadOracles} we construct a higher $\Delta^0_2$ real~$y$ below which higher Turing and fin-h reducibility differ. By \cref{prop:higher_Turing_when_wock_is_preserved}, the real~$y$ collapses $\wock$. By \cref{prop:collapsing_and_fin-h},~$y$ does not have a collapsing approximation.

\subsubsection{A real with a club approximation but no approximation which is finite-change along true initial segments}

This is a slightly finer diagonalisation argument. Let $\seq{x^n_t}$ be the array given by \cref{lem:listing_fcatis}. We build an approximation $\seq{y_t}_{t<\wock}$ and diagonalise against each~$x^n$ by showing that $y\rest{n+1}\ne x^{n}\rest{n+1}$, provided $\seq{x^n_t}$ changes finitely along true initial segments. 

To ensure that~$y$ has a club approximation we follow the construction of the proof of \cref{lem:club_and_almost_finite_change}. As in that construction, define the sets~$F_\s$, and the notion of a string being permitted or forbidden at stage~$s$. We again ensure that all initial segments of each~$x_t$ are permitted at stage~$t$ and that the approximation is locally continuous. 

At stage~$s$, given $\seq{y_t}_{t<s}$, define $y_s$ by recursion. We are given $\s = y_s\rest n$, which by induction is permitted at stage~$s$. Then:
\begin{enumerate}
	\item If an immediate extension $\s\conc i$ of~$\s$ is forbidden at stage~$s$, then we let $y_s(n) = 1-i$.
	\item If $s$ is a limit stage, $F_\s$ is cofinal in~$s$ and $x^n_t$ is a constant~$i$ on a final segment of $F_\s\cap s$, then we let $y_s(n)=i$ (note that the assumption implies that $\s\conc i$ is permitted at every stage $t<s$, and so also at~$s$).
	\item Otherwise, we consider the set $A_{\s} = \{ t<\wock\,:\, \s\prec x^n_t\}$. If $x^n_s(n)$ changes infinitely along the stages in $A_{\s}\cap s$ (there is an increasing sequence $\seq{t(i)}$ of stages $t(i)\in A_{\s}\cap s$ such that $x^n_{t(i+1)}(n)\ne x^n_{t(i)}(n)$ for all $i<\w$) then we let $y_s(n) = 0$. Otherwise, $y_s(n)$ is a constant~$i$ on a final segment of $A_{\s}\cap s$;\footnote{This includes the case that $A_{\s}\cap s$ has a greatest element $t$; then $i = x^n_t(n)$.} we let $y_s(n)=1-i$. 
\end{enumerate}

By construction, the sequence~$\seq{y_s}$ is locally almost finite-change, and so $y = \lim_s y_s$ has a club approximation. Let~$n<\w$ such that $\seq{x^n_t}$ changes finitely along true initial segments. Let $\s = y\rest n$. If $\s\ne x^n\rest n$ we are done, so we assume that $\s\prec x^n$ as well. The value $x^n_t(n)$ changes finitely often on~$A_\s$. By induction on $s<\wock$ we see that the value $y_t(n)$ changes only finitely often on~$F_\s$ and that both $\s\conc 0$ and $\s\conc 1$ are permitted at~$s$. We then succeed in ensuring that $y(n)\ne x^n(n)$.

\subsubsection{A real with a compact approximation but no club approximation}

To show that there are no more implications in \cref{fig:Delta_2_classes}, it remains to show that there is a real $x\in 2^\w$ which has a closed approximation but not a club approximation. Note that this also shows that there is a real which has a collapsing approximation but not a club approximation. 

To construct a real with a closed approximation we use the following.

\begin{lemma} \label{lem:sufficient_condition_for_compactness}
	Let $\seq{x_s}$ be a $\wock$-computable approximation of $x\in 2^\w$. Suppose that for all limit $s<\wock$ there are at most finitely many~$n$ such that $\lim_{t\to s} x_t(n)$ does not exist. Then~$x$ has a closed approximation.
\end{lemma}

\begin{proof}
	We use \cref{lem:compact_and_countable_closure}. Since $\wock$ is countable, it suffices to show that for all limit $s<\wock$ there are at most countably many $y$ which are the limit $\lim_{i\to \w} x_{t(i)}$ where $\seq{t(i)}$ is increasing and $s=\sup_k t(i)$. But the condition implies that for a fixed~$s$, all such~$y$ differ on only finitely many bits. 
\end{proof}

We in fact show the following. 

\begin{proposition} \label{prop:compact_and_uniformly_Delta2}
	No uniform listing of higher $\Delta^0_2$ elements of Cantor space contains all reals with closed approximations. That is, if $\seq{x^n_t}$ is a $\wock$-computable array such that for all~$n$, $\seq{x^n_t}$ converges to a real~$x^n$, then there is some $y\in 2^\w$ with a closed approximation which equals none of the $x^n$. 
\end{proposition}

We then use \cref{lem:listing_clubs} to obtain the desired separation. 

\medskip

To prove \cref{prop:compact_and_uniformly_Delta2} we will in fact build an approximation $\seq{y_t}$ such that for all limit $s<\wock$ there is at most one $k<\w$ such that $\lim_{t\to s} y_t(k)$ does not exist. 

The na\"ive approach, letting $y_t(n) = 1-x^n_t(n)$, will of course not work, since it is likely that for some limit $s<\wock$, $x^n_t(n)$ change infinitely often up to~$s$ for more than one~$n$. However we can choose other witnesses~$k$ to diagonalise~$y$ against~$x^n$. Adding bounded injury to the argument makes it work. 

In detail, along with $\seq{y_t}$ we also define a sequence of witnesses $k^n_t$ for all $n<\w$ and $t<\wock$. Witnesses for different~$n$ are distinct; this is achieved by requiring that $k^n_t \in \w^{[n]}$ (the $n\tth$ column of $\w$) for all~$n$. Once the witnesses $k^n_t$ are defined, $y_t$ is determined by letting:
\begin{itemize}
	\item $y_t(k^n_t) = 1 - x^n_t(k^n_t)$ for all $n<\w$; and
	\item $y_t(k) = 0$ if $k\ne k^n_t$ for all~$n$. 
\end{itemize}
The idea is that if we see $y$ change on $k^n_t$ then we discard $k^m_t$ for $m>n$. In detail: at stage~$s$, we need to define a new witness $k^n_s$ in case either
\begin{enumerate}
	\item $s$ is a limit stage and $k_t^n$ is not stable below~$s$ ($k_t^n$ is not constant on a final segment of~$s$); or
	\item for some $m<n$, it is not the case that $k=k_{<s}^m$ and $i=x_{<s}^m$ are well-defined and $x^m_s(k)=i$. In other words, either
	\begin{itemize}
		\item $s$ is a successor stage and for $k = k_{s-1}^m$ we have $x^m_{s-1}(k) \ne x^m_s(k)$; or
		\item $s$ is a limit stage, and either $k_t^m$ is not stable below~$s$; or it is, with value~$k$, but $x^m_t(k)$ is not stable below~$s$; or it is, with value~$i$, but $x^m_s(k)\ne i$.\footnote{We could omit the very last case by requiring that $\seq{x^n_t}$ is partially continuous.}
		\end{itemize}
	\end{enumerate}
In all cases, we let $k_s^n$ be the $p(s)\tth$ element of the column $\w^{[n]}$. If none of these cases hold, then we let $k_s^n= k_{<s}^n$, where as usual this means $k_{s-1}^n$ if~$s$ is a successor stage, or the stable value $k_I^n$ for some final segment~$I$ of~$s$ if~$s$ is a limit stage.

This concludes the construction. By induction on~$n$ we see that each $k^n_t$ reaches a limit~$k^n$ and that $y(k^n)\ne x^n(k^n)$. It remains to show the condition which implies compactness. Let~$s<\wock$ be a limit stage. Suppose that there is $n<\w$ such that $k^n_t$ is stable on a final segment~$I$ of~$s$ (with value~$k^n_s$), but that $\lim_{t\to s} x^n_t(k^n_s)$ does not exist. For all $m<n$, both $k^m_t$ and $x^m_t(k^m_t)$ are stable on~$I$. If $k\ne k^m_s$ for all $m\le n$ then on a final segment of~$s$, $k\ne k^m_t$ for all $m<\w$ (if it is ever chosen, it is discarded before stage~$s$), and so $y_t(k)=0$ on a final segment of~$s$. This shows that $\lim_{t\to s} y_t(k)$ exists for all $k\ne k^n_s$.

If no such~$n$ exists, then by induction on~$n$ we see that both $k^n_t$ and $x^n_t(k^n_t)$ are stable below~$s$ (though likely there is no single final segment~$I$ of~$s$ on which they are all stable). Thus if $k = k^n_s$ for some~$n$, then $\lim_{t\to s} y_t(k)$ exists. Suppose that $k\ne k^n_s$ for all~$n$. Say $k\in \w^{[n]}$. Then $k\ne k^m_t$ for all $m\ne n$ and all~$t$; and $k\ne k^n_t$ on a final segment of~$s$, so again $y_t(k)=0$ on a final segment of~$s$.

\subsection{A remark on club approximations}

We can weaken \cref{def:club_approximation} as follows.

\begin{definition} \label{def:club-quasi-approximation}
	A $\wock$-computable sequence $\seq{f_s}_{s<\wock}$ is a \emph{club quasi-approximation} of a function~$f$ if for all~$n<\w$, the set of stages~$s$ at which $f\rest n = f_s\rest n$ is a closed and unbounded subset of~$\wock$. 
\end{definition}
The point is that we do not require that $f = \lim_s f_s$. If~$\seq{f_s}$ is a club approximation of any function, then this function is determined uniquely: for each string~$\s$, $\{s\,:\, \s\prec f_s\}$ is a $\wock$-computable set, and the intersection of finitely many $\wock$-computable club subsets of~$\wock$ is a club subset of~$\wock$. 

For elements of Cantor space we get nothing new: if~$x\in 2^\w$ has a club quasi-approximation then it has a club approximation, in particular it is higher~$\Delta^0_2$. However there are elements of Baire spaces which have club quasi-approximations but are not higher $\Delta^0_2$. 

To see this, following the discussion in \cref{subsec:enumerating_approximations}, fix a total $\wock$-computable array $\seq{f^n_t}$ of functions which contains all $\wock$-computable approximations. We define a sequence $\seq{g_t}_{t<\wock}$ which is a club quasi-approximation of $g\in \w^\w$, ensuring that if $\seq{f^n_t}$ converges to some $f^n\in \w^\w$ then $g(n)\ne f^n(n)$. In fact we will ensure a stronger property than required: for all~$n$, the set of stages~$t<\wock$ such that $g_t(n)= g(n)$ is closed and unbounded. The definition is simple: at a limit stage~$s$ we let $g_s(n) = \lim_{t\to s} g_t(n)$ if the limit exists, and~$0$ otherwise. At a successor stage~$s$ we compare $g_{s-1}(n)$ and $f^n_s(n)$. If they are distinct we let $g_s(n) = g_{s-1}(n)$. If they are equal to a nonzero value, we let $g_s(n) = 0$. If they are both equal to 0 then we let $g_s(n) = p(s)$, where as usual $p\colon \wock\to \w$ is $\wock$-computable and injective. Now the point is that for all $k\ne 0$, the set of stages $\{ t<\wock\,:\, g_t(n)=k\}$ is an interval of stages and so closed; and that the set of stages $\{ t<\wock\,:\, g_t(n) = 0 \}$ is closed. By admissibility of $\wock$, one of these sets must be unbounded.

\medskip

Finally we remark that the proof of the second part of \cref{lem:club_approximations-are_collapsing_etc} (that every club approximation of $f\notin\Delta^1_1$) shows that every club quasi-approximation of $f\notin \Delta^1_1$ is ``quasi collapsing'' in that the sequence of stages~$s(n)$ at which we first observe $f\rest n$ is unbounded in~$\wock$. Hence if $f\notin \Delta^1_1$ has a club quasi-approximation  then $\w^f>\wock$, even if~$f$ is not higher $\Delta^0_2$.

\section{The class $\MLR[O]$} \label{section:mlr_plop_o} 


It is not very hard to prove that one can characterize weak 2 randomness using a restricted relativisation of ML-randomness to~$\emptyset'$. Define an $\MLR[A]$-test to be a nested test $\seq{U_n}$ satisfying $\leb(V_n)\le 2^{-n}$, where each~$U_n$ is \emph{effectively} open (not $A$-effectively open), but an index for each~$U_n$ is given by~$A$. That is, $U_n = W_{f(n)}$ where $\seq{W_e}$ enumerates effectively open sets and~$f\le_\Tur A$. We then have weak 2 randomness is equivalent to $\MLR[\emptyset']$-randomness. 

 One direction is straightfoward; given a weak 2 test $\seq{V_n}$, $\emptyset'$ can find the least~$m$ such that $\leb(V_m)\le 2^{-n}$. The other direction requires a time-trick: if that $\seq{W_{f(n)}}$ is a test as described then we cover it with the null~$\Pi^0_2$ set $\bigcap_{n,t} \bigcup_{s>t} W_{f_s(n)}$. Trying to lift the argument to the higher setting fails since the intersection would be over $\w\times \wock$-many higher open sets, and we have no way to effectively covert this to an $\w$-list. 

\smallskip

We shall indeed prove that the notion of higher Martin-L\"of randomness, where Kleene's~$O$ can be used for the index of each component is much stronger than higher weakly 2-randomness, and even stronger than~$\Pi^1_1$-randomness. We now let $\seq{W_e}$ enumerate the \emph{higher} effectively open sets. 

\begin{definition}
Let $A\in 2^\w$. A \emph{higher $\MLR[A]$-test} is a nested sequence $\seq{W_{f(n)}}$ where $f\le_{\hT}A$ and $\leb(W_{f(n)})\le 2^{-n}$. The null set determined by such a test is $\bigcap_n W_{f(n)}$. A sequence is in $\MLR[A]$ if it is not captrued by any $\MLR[A]$-test.
\end{definition}

Of course for Kleene's~$O$ the index-function~$f$ can be taken to be $O$-computable (\ref{prop:O_collapses_higher_computability}); however the building blocks are still higher effectively open sets.

\smallskip

We start by giving an alternate characterisation of $\MLR[O]$. A \emph{long (higher) ML-test} is a sequence $\seq{U_\alpha}_{\alpha<\wock}$ of uniformly higher effectively open sets such that $\bigcap_{\alpha}U_\alpha$ is null. No assumption is made about nesting. 

\begin{lemma} \label{lem_mlrplop_equ}
Higher $\MLR[O]$ tests and long ML-tests capture the same null sets. 
\end{lemma}

\begin{proof}
One direction follows the failed time trick: if $\bigcap_n W_{f(n)}$ is an $MLR[O]$ test then for $n<\w$ and $s<\wock$ we let $V_{n,t} = \bigcup_{s>t} W_{f_s(n)}$. We can reorder the array $\seq{V_{n,s}}$ effectively in ordertype $\wock$ using an effective bijection between $\w\times \wock$ and~$\wock$. If~$t$ is sufficiently late then $V_{n,t} = W_{f(n)}$. 

In the other direction let $\seq{U_s}_{s<\wock}$ be a long ML-test. Using~$O$, for each~$n$ we can find a finite set $F\subset \wock$ such that $\leb(\bigcap_{s\in F}U_s)\le 2^{-n}$ (the measure of a higher effectively open set is $O$-computable, uniformly).
\end{proof}

Hirschfeldt and Miller (see~\cite{DowneyH2010}) showed that a ML-random sequence is weak 2 random if and only if it forms a minimal pair with~$\emptyset'$; the witness for failure of this property can be taken to be c.e. The situation is more complicated in the higher setting. Higher weak 2 randomness does not seem to align with such a property. In~\cite{Pi11RandomnessPaper} the authors show that $\Pi^1_1$-randomness partly corresponds to this property: a higher ML-random sequence~$X$ is $\Pi^1_1$-random if and only if there is no higher-c.e., non hyperarithmetic set higher Turing reducible to~$X$. However, not every $\Pi^1_1$-random sequnece forms a minimal pair with Kleene's~$O$ in the higher Turing degrees; by the Gandy basis theorem, there is a $\Pi^1_1$-random sequence computable from~$O$. 

Higher $\MLR[O]$ gives a certain analogue of the Hirschfeldt-Miller property. Recall that we extended the notion of higher Turing reducibility to subsets of~$\wock$ in the obvious way. 

\begin{proposition}
The following are equivalent for a higher ML-random sequence~$X$:
\begin{enumerate}
\item $X\notin \MLR[O]$.
\item $X$ higher Turing computes a $\wock$-c.e.\ subset of~$\wock$ which is not $\wock$-computable. 
\item $X$ higher Turing compute a $\Delta_2$ subset of~$\wock$ which is not $\wock$-computable. 
\item \label{item:ce-fying} There is some higher~$\Delta^0_2$ subset of~$\w$ which is not higher c.e., but is higher c.e.\ in~$X$. 
\end{enumerate}
\end{proposition}

We note that the lower setting analogue of property~(\ref{item:ce-fying}) does characterise weak 2 randomness, a fact which has not been observed so far.

\begin{proof}
(1)$\then$(2): the lowercase argument can be copied to the higher setting. Let $\seq{V_\alpha}$ be a long ML-test capturing~$X$. Using an indexing of all finite subsets of~$\wock$ (and taking finite intersections) we may assume that for all $\epsilon>0$, there are unboundedly many~$\alpha$ such that $\leb(V_\alpha)< \epsilon$. We enumerate a $\wock$-c.e.\ subset~$A\subseteq \wock$, attempting to meet the requirements $P_\beta$: the complement of~$A$ is not $W_\beta$, where $\seq{W_\beta}$ is a $\wock$-effective sequence of all $\wock$-c.e.\ subsets of~$\wock$. Suppose that a requirement~$P_\beta$ has not been initialised since stage~$t<\wock$, is not yet met at stage~$s>t$, and that at stage $s>t$ we see that some $\alpha\in W_{\beta,s}$ and $\leb(V_{\alpha,s})\le 2^{-p(\beta)}$ for some $\alpha\in {\wock}^{[\beta]}$. Then we enumerate~$\alpha$ into~$A_{s+1}$ and initialise every requirement $P_\gamma$ where $p(\gamma)>p(\beta)$. We also let $G_\alpha = V_{\alpha,s}$. If $\alpha\notin A$ then we let $G_\alpha = \emptyset$. Then $\seq{G_\alpha}$ is a higher Solovay test, and if~$X$ is not captured by this test then $A\le_{\hT}X$. 

(3)$\then$(4): Say $B\le_{\hT} O$ is not $\wock$-computable and that $B\le_{\hT}X$. Then $p[B]$ is higher $X$-c.e.\ but is not hyperarithmetic. 

(4)$\then$(1): Let $C\subset \w$ be $O$-computable, not higher c.e., but higher $X$-c.e. The usual majority-vote argument shows that the set of oracles~$Y$ such that $C$ is higher $Y$-c.e.\ is null. Let $\seq{C_s}_{s<\wock}$ be a $\wock$-computable approximation of~$C$, and let~$\Gamma$ be a higher enumeration functional. For $n,k<\w$ and $t<\wock$ let $V_{n,k,t}$ be the set of $Y\in 2^\w$ such that for some $s>t$, either:
\begin{itemize}
	\item $n\in C_s$ and $n\in \Gamma_s^X$; or
	\item $n\notin C_s$ and $n\notin \Gamma_s^{X\rest{k}}$. 
\end{itemize}
Then $\seq{V_{n,k,t}}$ is a long ML-test which captures the oracles~$Y$ such that $\Gamma^Y =\nobreak C$. 
\end{proof}

Finally we show that higher $\MLR[O]$-randomness is strictly stronger than $\Pi^1_1$-randomness.

\begin{proposition}
Higher $\MLR[O]$-randomness is strictly stronger than $\Pi^1_1$-randomness.
\end{proposition}

\begin{proof}
As mentioned before, there is an $O$-computable $\Pi^1_1$-random sequence; no higher $\MLR[O]$-random sequence can be $O$-computable. 

Suppose that~$X$ is not $\Pi^1_1$-random; we show it is not higher $\MLR[O]$-random. We assume that~$X$ is higher ML-random. By \cite{Monin2014}, there exists a uniformly higher effectively open sequence $\seq{U_n}$ such that $X \in \bigcap_n \U_n$ but $X$ is not an element of any higher effectively closed set~$F\subseteq \bigcap_n U_n$. The set of canonical indices of higher effectively closed subsets of $\bigcap_n U_n$ is higher c.e.; this gives us a sequence $\seq{P_{\alpha}}_{\alpha<\wock}$ which enumerates the higher effectively closed subsets of~$\bigcap_n U_n$. Then the sequence $\seq{U_n}$ together with the sequence of the complements of the $P_{\alpha}$'s gives a long ML-test which captures~$X$. 
\end{proof}

\section{Higher Oberwolfach randomness (with Dan Turetsky)} \label{sec:OW}

Oberwolfach randomness \cite{OWpaper} is the notion of randomness which captures computing all $K$-trivials: a ML-random sequence computes all $K$-trivial sets if and only if it is not Oberwolfach random. The higher analogue holds.

\begin{definition}
A \emph{higher Oberwolfach test} is a pair $(\seq{G_\s},\alpha)$ where:
\begin{itemize}
	\item For $\s\in 2^{<\w}$, $G_\s$ is (uniformly) higher effectively open, and $\leb(G_\s)\le 2^{-|\s|}$;
	\item The array is nested, in the sense that if $\s\preceq \tau$ then $G_\tau\subseteq G_\s$; and
	\item $\alpha\in 2^\w$ is a higher left-c.e. sequence.
\end{itemize}
The null set determined by the test is $\bigcap_{n<\w}G_{\alpha\rest n}$. A sequence is higher Oberwolfach random if it is not captured by any higher Oberwolfach test.
\end{definition}

Proposition~\ref{prop:Demuth_characterisations_of_weak_2_randomness} shows that every higher weak 2 random sequence is higher Oberwolfach random; higher difference randomness can be characterised using ``version-disjoint'' higher Oberwolfach tests and so higher Oberwolfach randomness implies higher difference randomness (this follows from the proof of one of the implications in~\cref{prop:higher_limit_lemma}, and is identical to the lower setting). In fact both implications are strict. It is not difficult to build a higher Oberwolfach random sequence with a compact approximation, and then appeal to~\cref{prop:w2r_closed_approx} to separate between higher weak 2 randomness and higher Oberwolfach randomness. To separate between higher Oberwolfach randomness and higher difference randomness we need to appeal to the forcing used by Day and Miller~\cite{DayMiller2013} to construct a difference random set which is not a density one point in effectively closed sets; the argument can be performed in the higher setting without change, both constructing such a random and showing that such a random cannot be higher Oberwolfach random. 

\smallskip

The characterisation of higher Oberwolfach randomness in terms of computing $K$-trivial sets consists of two steps:

\begin{theorem} \label{thm:not_OW_computes_K_trivials}
If $X$ is higher ML-random but not higher Oberwolfach random, then it higher Turing computes every higher $K$-trivial set.
\end{theorem}

\begin{theorem} \label{thm:smart_K_trivial}
There is a higher $K$-trivial set which is not higher computable from any higher Oberwolfach random sequence. 
\end{theorem}

A set~$A$ as given by \cref{thm:smart_K_trivial} is called a ``smart'' $K$-trivial set: any higher ML-random sequence which higher computes~$A$, must higher compute all higher $K$-trivial sets. 

\

The usual proof of the lower-setting analogue of \cref{thm:not_OW_computes_K_trivials} passes through a characterisation of Oberwolfach randomness in terms of weak 2 tests which are bounded by \emph{additive cost functions}. These are weak 2 tests $\seq{U_n}$ whose measure is bounded by $\alpha-\alpha_n$, where $\seq{\alpha_n}$ is an increasing approximation of a left-c.e.\ real~$\alpha$. By their very definition these use a time-trick. We can emulate the time trick by working over a $K$-trivial oracle. 

\begin{proof}[First proof of \cref{thm:not_OW_computes_K_trivials}]
 Let~$X$ be higher ML-random but not higher Oberwolfach random. Since~$X$ is not higher weak 2 random, the Hirschfeldt-Miller argument shows that there is some non-hyperarithmetic, higher c.e.\ set~$B$ which is higher Turing reducible to~$X$ (in fact this is true for any higher ML-random which is not $\Pi^1_1$-random). We may assume that~$X$ is higher difference random, and so the Hirschfeldt-Nies-Stephan argument shows that~$B$ is higher $K$-trivial. The idea is to work relative to~$B$ and emulate the proof in~\cite{OWpaper}.

	Since~$B$ has a collapsing approximation, working relative to~$B$ we can revert to computability of length~$\w$ (see Lemma~\ref{lem:collapsing_implies_Schnorr-Levin} and its footnote). Let~$\seq{g(n)}$ be an increasing, cofinal sequence in $\wock$ which is $\wock$-computable from~$B$; let $(\seq{G_\s},\alpha)$ be a higher Oberwolfach test capturing~$X$. We let $U_n = \bigcup_{s\ge g(n)} G_{\alpha_s\rest k}$. Then $\seq{U_n}$ is nested and uniformly higher $B$-c.e.; and $\leb(U_n)\le 2^{-n} + (\alpha-\alpha_{g(n)})$. By delaying the approximation of $U_n$ we can also suppose $\leb(U_{n, g(m)})\le 2^{-n} + (\alpha_{g(m)}-\alpha_{g(n)})$ for each $n$ and $m$.

	Let $\cost(k,s) = \alpha_s - \alpha_{g(k)}$. The aim is to find a higher $B$-computable approximation $\seq{A_n}_{n<\w}$ of~$A$ such that letting $k(n) = |A_{n-1}\wedge A_n|$ (the least~$k$ such that $A_{n}(k)\ne A_{n-1}(k)$), we have $\sum_{n<\w} (\alpha_{g(n)}-\alpha_{g(k(n))})$ is finite (we may assume that $k(n)\le n$; otherwise we replace $\alpha_{g(n)}-\alpha_{g(k(n))}$ by 0). Once we have such an approximation we can define a higher $B$-Solovay test $\seq{G_k}$ by letting $G_k = U_{k,g(n)}$ if~$n$ is the greatest such that $k = k(n)$ (and $G_k = \emptyset$ if there is no such~$n$). Since~$B$ is higher $K$-trivial, $X$ cannot be captured by this test, and then the usual argument builds a higher $B$-c.e.\ functional $\Phi$ such that $\Phi(X)=A$. Since~$B\le_{\hT}X$ we get $A\le_{\hT} X\oplus B \le_{\hT} X$ as required. 

	\smallskip

	To obtain the required approximation $\seq{A_n}$ we can operate in two ways. We define the higher $B$-c.e.\ oracle discrete measure $\mu^\tau(n) = \alpha_{g(n+1)}-\alpha_{g(n)}$ (for all strings~$\tau$ of length~$n$). One way is to use the fact that~$A$ is higher $K$-trivial relative to~$B$; we repeat the proof of the main lemma in $(L_{\wock};\in, B)$ and use it for the measure $\mu^A$. Another way is to directly use the unrelativised main lemma (\cref{prop:main_lemma_of_K_triviality}). 
	Recall that we can let $g(n)$ be the least~$s$ such that $B_s\rest n = B\rest n$ for some fixed higher enumeration $\seq{B_s}_{s<\wock}$ of~$B$. For $t<\wock$ we let $g_t(n)$ be the least $t\le s$ such that $B_s\rest n = B_t\rest n$. Note that $\sup_n g_s(n) = s$ and that $g_s(n)\le g(n)$. For all~$\tau$ of length~$n$ we let $\mu_s^{B_s\rest n \oplus \tau}(n) = \alpha_{g_s(n+1)}- \alpha_{g_s(n)}$. Let $\seq{A_s}$ be a collapsing approximation of~$A$. The main lemma gives us a $\wock$-computable closed and unbounded set~$C\subseteq \wock$, such that the sum $\sum_{s\in C} (\alpha_{s} - \alpha_{g_s(k(s))}) $ is finite; here $k(s) = |A_s(k)\wedge A_{s^+}(k)|$, where $s^+$ is the next element of~$C$ beyond~$s$. We define the required $B$-computable approximation of~$A$ by letting $\hat A_n = A_{s(n)}$ for some $s(n)\in C$, $s(n)\ge g(n)$ (for example $s(n) = \min (C\setminus g(n))$). Let~$k = |\hat A_{n-1}\wedge \hat A_n|$. Then there is some $s\in [s(n), s(n+1))$ such that $k\ge |A_s\wedge A_{s^+}|$. Since $s\ge g(n)$, $\alpha_{g(n)} - \alpha_{g(k)} \le \alpha_{s} - \alpha_{g_{s}(k)}$.
\end{proof}

We can however eliminate the time trick, with an argument which also works in the lower setting. Rather than use additive cost functions, we use cost functions which in the lower setting are ``subadditive''. If~$\mu$ is a discrete measure then we let $\cost_\mu(n) = \sum_{m\ge n} \mu(m)$. If $\seq{\mu_s}$ is an increasing enumeration of a left-c.e.\ discrete measure~$\mu$ then we let $\cost_\mu(n,s) = \sum_{m\ge n} \mu_s(m)$. We say that an approximation $\seq{A_s}_{s<\wock}$ of a set~$A$ witnesses that~$A$ \emph{obeys}~$\cost_\mu$ if the sum $\sum_{s<\wock} \cost(|A_s\wedge A_{s+1}|,s)$ is finite. If~$\+\mu$ is the optimal left-c.e.\ discrete measure then any set obeying~$\cost_{\+\mu}$ must be higher $K$-trivial. If~$A$ is higher $K$-trivial then the main lemma (\cref{prop:main_lemma_of_K_triviality})) shows that $A$ obeys~$\cost_\mu$ for any left-c.e.\ discrete measure~$\mu$. 

A \emph{$\cost_\mu$-bounded test} is a higher weak 2 test $\seq{U_n}$ such that $\leb(U_n) \le^\times \cost_\mu(n)$; if $\seq{U_n}$ is such a test then we may assume that $\leb(U_{n,s})\le^\times \cost_\mu(n,s)$ (where of course the multiplicative constant is the same for all~$s$). The usual argument shows that if $X$ is a higher ML-random sequence which is captured by some $\cost_\mu$-bounded test and~$A$ obeys $\cost_\mu$ then $A\le_{\hT}X$. So \cref{thm:not_OW_computes_K_trivials} follows from:

\begin{proposition} \label{prop:covering_by_Auckland}
	A sequence is higher Oberwolfach random if and only if it is not captured by any $\cost_{\+\mu}$-bounded test. 
\end{proposition}

\begin{proof}
In one direction, let~$(\seq{G_\s},\alpha)$ be a higher Oberwolfach test. For all $n<\w$ and $s<\wock$, let:
\begin{itemize}
	\item $k_{n,s} = \# \left\{ \alpha_t\rest n \,:\,  t\le s \right\}$ and
	\item $m_{n,s}$ be the integer~$m$ such that $m2^{-n} \le \alpha_s < (m+1)2^{-n}$. 
\end{itemize}
We define a higher left-c.e.\ discrete measure~$\nu$ with the aim that $\cost_\nu(n,s) = 2^{-n}k_{n,s} + (\alpha_s - 2^{-n}m_{n,s})$. We would then let $U_n = \bigcup_{s<\wock} G_{\alpha_s\rest n}$; $\leb(U_n) = 2^{-n} k_{n,\wock} \le \cost_{\nu}(n)$. The measure~$\nu$ is not difficult to define. We may assume that for limit~$s$, $\alpha_s = \lim_{t\to s} \alpha_t$ and so we can let $\nu_s = \sup_{t<s}\nu_t$. Let $\s = \alpha_s\wedge \alpha_{s+1}$. We may assume that $\alpha_{s+1} = \s 1  0^\w$. We then let $\nu_{s+1}(n) = \nu_s + 2^{-n}$ if $n> |\s|+1$ and $\alpha_s(n)=0$; otherwise we let $\nu_{s+1}(n) = \nu_s(n)$. 

\smallskip

In the other direction let $\seq{U_n}$ be a $\cost_{\+\mu}$-bounded test; say $\leb(U_n)\le d\cdot \cost_{\+\mu}(n)$. Let~$\mu = d\cdot \+\mu$ (so $\leb(U_n)\le \cost_{\mu}(n)$). By taking a tail of the measure~$\mu$ (and of the test) and renumbering, we may assume that $\mu(\w)<1$. We let $\alpha_s = \cost_{\mu}(0,s) = \mu_s(\w)$. We define indices $k_s(n)$ for $n<\w$ and $s<\wock$; we let $G_{\alpha_s\rest n,s} = U_{k_s(n),s}$. To keep the sets~$G_\s$ nested we ensure that $k_s(n)$ is increasing in~$n$. We redefine $k_s(n)$ if $\alpha_{s}\rest n \ne \alpha_{s-1}\rest n$. To redefine it we pick a new value~$k$ such that $\cost_{\mu}(k,s)\le 2^{-n}$. Let $t<\wock$ and let $\s = \alpha_t\rest n$; let~$s$ be the least stage such that $\s\prec \alpha_s$; let $k = k_s(n)= k_t(n)$. We claim that $\leb(G_{\s,t})\le 2^{-(n-1)}$. For $G_{\s,t} = U_{k,t}$ and $\leb(U_{k,t})\le \cost_{\mu}(k,s)$; if this is greater than $2\cdot 2^{-n}$ then as $\cost_{\mu}(k,s)\le 2^{-n}$ we have $\alpha_t - \alpha_s \ge \cost_{\mu}(k,t) - \cost_{\mu}(k,s) > 2^{-n}$; this implies that $\alpha_t\rest n \ne \alpha_s \rest n$. 
\end{proof}

The proof in \cite{OWpaper} constructing a smart $K$-trivial set works with subadditive, rather than only with additive cost functions. This proof can be adapted to the higher setting using the usual techniques for overcoming topological problems. However to prove \cref{thm:smart_K_trivial} we use a streamlined argument by Turetsky.  

\begin{proof}[Proof of \cref{thm:smart_K_trivial}]
	Let~$\Gamma$ be a ``universal'' higher Turing functional; $\Gamma(0^e1X) = \Phi_e(X)$. Since higher Oberwolfach randomness is invariant under the shift, it suffices to enumerate a higher $K$-trivial c.e.\ set~$A$ and a $\cost_{\+\mu}$-bounded test~$\seq{U_n}$ which captures every sequence~$X$ such that $\Gamma(X)=A$. In this proof let $\cost = \cost_{\+\mu}$.

	We may assume that for all~$n$, $\cost(n,0)> 0$. We enumerate~$A$ and $\seq{U_n}$ as follows. At each stage we have a ``follower'' $x_{n,s}$; the sequence $\seq{x_{n,s}}$ increases with~$n$. We also enumerate a global error set~$\EE_s$; $\EE_s$ is the set of oracles~$X$ such that $\Gamma_s(X)$ lies to the left of~$A_s$. Let 
	\[ G_{n,s} = \left\{ X \,:\,  \Gamma_s(X) \succeq A_s\rest{x_{n,s}+1} \right\}.
	\]
	We will have $U_{n,s}\subseteq \EE_s\cup G_{n,s}$. We will change $x_{n,s}$ only finitely many times (for each~$n$), and so at limit stages we can take limits of all objects. We ensure that $\leb(U_{n,s}\setminus \EE_s)\le \cost(n,s)$. Let~$s$ be a stage and let $n<\w$. If $\leb(G_{m,s}\setminus \EE_{s})\le \cost(m,s)$ for all $m\le n$ then we let $U_{n,s+1} = U_{n,s}\cup G_{n,s}$. If~$n$ is least such that $\leb(G_{n,s}\setminus \EE_{s})> \cost(n,s)$ then we enumerate $x_{n,s}$ into~$A_{s+1}$; we cancel $x_{m,s}$ for all $m>n$; for all $m\ge n$, we choose unused $x_{m,s+1}>m$ for $m\ge n$, 
	and let $U_{m,s+1} = U_{m,s}$. Note that the enumeration of~$x_{n,s}$ into~$A_{s+1}$ means that $U_{m,s+1}\subseteq \EE_{s+1}$ for all $m\ge n$.

	\smallskip

	The fact that $G_{n+1,s}\subseteq G_{n,s}$ (as $x_{n+1,s}> x_{n,s}$) ensures that $U_{n+1,s}\subseteq U_{n,s}$ for all~$n$ (and all~$s$). If $x_{n,s}$ is enumerated into~$A_{s+1}$ then $\leb(\EE_{s+1}\setminus \EE_{s})> \cost(n,s)\ge \cost(n,0)$. This shows that $x_{n,s}$ is enumerated into~$A_{s+1}$ at only finitely many stages~$s$. In turn this shows that~$x_{n,s+1}\ne x_{n,s}$ for only finitely many stages~$s$. 

	The enumeration $\seq{A_s}$ witnesses that~$A$ obeys~$\cost$, and so is higher~$K$-trivial. To see this, suppose that~$x=x_{n,s}$ is enumerated into~$A_{s+1}$. Then $\cost(x,s)\le \cost(n,s)$; 
	this shows that the total cost paid along this enumeration is bounded by $\leb(\EE)$. 

	Finally we need to show that $\leb(U_n)\le^\times \cost(n)$. We enumerate a left-c.e.\ measure~$\nu$, with the aim of having $\leb(\EE_s\cap U_{n,s})\le \cost_{\nu}(n,s)$ for all~$n$ and~$s$. We would then have $\leb(U_{n}) \le \cost_{\nu}(n) + \cost(n) \le^\times \cost(n)$ as required. At stage~$s$ we need to have
	 \[ \cost_\nu(n,s+1)-\cost_\nu(n,s)\ge \min \{\cost(n,s), \leb(\EE_{s+1}\setminus \EE_{s})\};\] this suffices since $U_{n,s+1}\cap(\EE_{s+1}\setminus \EE_{s}) \subseteq U_{n,s}\setminus \EE_{s}$. Since $\cost(n,s)\to 0$ as $n\to \w$ we can distribute a total of $\leb(\EE_{s+1}\setminus \EE_{s})$ among the natural numbers (so that $\nu_{s+1}(\w) \le \nu_{s}(\w) + \leb(\EE_{s+1}\setminus \EE_{s})$) to achieve the desired increase in $\cost_{\nu}(n,s+1)$. Of course $\nu$ is indeed a discrete measure since $\nu(\w)= \leb(\EE)$. 
\end{proof}

 \bibliographystyle{alpha}
 \bibliography{continuous_higher_randomness}

\end{document}